\documentclass[UTF8,a4paper,10pt]{article}

\usepackage{fancyhdr}
\usepackage{multicol} 
\usepackage{lastpage} 
\usepackage{geometry} 
\usepackage[subfigure,AllowH]{graphfig} 
\usepackage{amsmath}
\usepackage{dsfont}
\usepackage{mathtools}
\usepackage{framed}
\usepackage{amsfonts}
\usepackage{amssymb}
\usepackage{amsthm} 
\usepackage{theoremref}

\fancypagestyle{plain}{
	\fancyhf{}
	\rhead{Page \thepage\ of \pageref{LastPage}}
	\lfoot{}
	\cfoot{}
	\rfoot{}}
\newtheorem{theorem}{Theorem}[section]

\newtheorem{corollary}{Corollary}[section]
\newtheorem{proposition}{Proposition}[section]
\newtheorem{lemma}{Lemma}[section]
\newtheorem{definition}{Definition}[section]

\newtheorem{remark}{Remark}[section]

\newcommand{\NN}{\mathbb{N}} 
\newcommand{\RR}{\mathbb{R}} 

\newcommand{\Bcal}{\mathcal{B}}

\newcommand{\Lcal}{\mathcal{L}}

\newcommand{\ep}{\epsilon}

\makeatletter
\newcommand{\tpitchfork}{%
	\vbox{
		\baselineskip\z@skip
		\lineskip-.52ex
		\lineskiplimit\maxdimen
		\m@th
		\ialign{##\crcr\hidewidth\smash{$-$}\hidewidth\crcr$\pitchfork$\crcr}
	}%
}
\makeatother

\title{\textbf{\huge{Spreading speeds for Fisher-KPP equations with slowly decaying initial data in an almost periodic setting}}}
\author{\ Xing Liang, Linfeng Xu, Tao Zhou }
\date{} 

\begin{document}
	\newcommand{\supercite}[1]{\textsuperscript{\cite{#1}}}
	\maketitle
	
	\setlength{\oddsidemargin}{ 1cm} 
	\setlength{\evensidemargin}{\oddsidemargin}
	\setlength{\textwidth}{13.50cm}
	\vspace{-0.2cm}
	\begin{center} 
		\parbox{\textwidth}{ 
			{\textbf{Abstract}} \quad {This paper investigates the long-time behavior of the Fisher-KPP equation with slowly decaying initial data in an almost periodic medium. 
	  We mainly focus on two classes of initial data:  exponentially decaying initial data and initial data that decay more slowly than any	 exponential function.
	Employing the Hamilton–Jacobi approach, we provide a unified framework for analyzing the Cauchy problems with initial data in both cases.
			We demonstrate that the level sets of the solution can be estimated by the generalized principal eigenvalue of the linearized operator and the decay rate of the initial data.}

			} 
	\end{center}
	
	\vspace{0.5cm}
	%
	%
	%
	%
	%
	%

	\setlength{\oddsidemargin}{-.5cm} 
	\setlength{\evensidemargin}{\oddsidemargin}
	\setlength{\textwidth}{17.00cm}
	
	\section{Introduction and main results}
	
	In this paper, we study the spreading speeds associated with the following equation:
	\begin{equation}\label{eq:main eq}
			\begin{cases}
	u_{t}=\partial_x(a(x)\partial_xu)+b(x)\partial_xu+f(x,u), \ &t>0,\ x\in\mathbb R,\\
	u(0,x)=u_0(x)\in[0,1], \  &x\in\mathbb R.\\
	\end{cases}
	\end{equation}
	We impose the following assumptions throughout the paper:\\
	(\textbf{A1}) The coefficients $a\in C^{1,\alpha}(\RR) \text{ and } b\in C^{\alpha}(\RR)$ are uniformly H\"{o}lder continuous with exponent  $\alpha\in(0,1)$. Moreover, $a \text{ and }b$ are almost periodic with $\inf\limits_{x}a(x)>0.$\\
     The reaction term $f$  is assumed to satisfy the      following conditions:\\
     (\textbf{F1})  $f$ is of class $C^1$ in $s$,  $\partial_sf(\cdot,0)\in C^{\alpha}(\mathbb R)\text{ is almost periodic with }\inf\limits_{x}\partial_sf(x,0)>0.$\\
     (\textbf{F2}) $f$ is assumed to be of the Fisher–KPP type, that is
    \begin{equation}
    	f(x,0)=f(x,1)=0,\ 0<f(x,s)\leq\partial_sf(x,0)s \text{ for any }  s\in(0,1).
    \end{equation}
    (\textbf{F3}) there exist $\beta>0, s_0\in(0,1)$ and $C\geq0$ such that 
    \begin{equation}\label{eq:f is c alpha}
    	f(x,s)\geq\partial_sf(x,0)s-Cs^{1+\beta} \text{ for any }  s\in(0,s_0).
    \end{equation}
    For example, a concave function $f(x,\cdot)\in C^{1,\beta}([0,1])$ uniformly in $x$, which is positive 	in $(0, 1)$ and vanishes at $0$ and $1$, satisfies the conditions above.

	Equation \eqref{eq:main eq}  represents a heterogeneous generalization of the classical homogeneous equation:
	$$	u_{t}=u_{xx}+f(u),$$
	where $f(0)=f(1)=0$ and $f(s)>0$ for $s\in(0,1)$, which was  studied by Fisher \cite{F1}, and  Kolmogorov, Petrovsky, and Piskunov  \cite{K1}.	
  A central issue is to investigate the evolution  of the level sets of solutions to equation \eqref{eq:main eq} under different types of initial data. Specifically, for any $\theta\in(0,1)$, we want to estimate	the following two quantities for sufficiently large  $t$:
  \begin{equation*}
  x_\theta^+(t)=\max E_\theta(t) \text{ and } x_\theta^-(t)=\min E_\theta(t),
   \end{equation*}
  where $E_\theta(t):=\{x| u(t,x)=\theta\}$ is the so-called  level set  of $u$ of value $\theta$ at time $t$.
  Consider the solution of equation \eqref{eq:main eq}  with compactly supported initial data. Existing results show that, under appropriate conditions, the limit 
  $\lim\limits_{t\to+\infty}\frac{x_{\theta}^{\pm}(t)}{t}$ exists and is independent of $\theta$, see, e.g., \cite{BHN,BHN2,berestycki2008Asymptotic,berestycki2012spreading,berestycki2019asymptotic,evans-souganidis2,GF,Huang-Shen09,LiangZhao07CPAM,Shen10TAMS,W} and references therein. 
 The existence and characterization of the spreading speed have already been proved. 
 As the present paper is concerned with propagation under initial data exhibiting different decay rates, we introduce the following definition of the spreading speed associated with initial data:  
	\begin{definition}
		Let $\mathcal U$ be a set of functions defined on $\mathbb R$ and $u$ be a solution to \eqref{eq:main eq} with initial datum $u_0\in\mathcal U$. 
		If there exist two constants $\omega^{+}\in\mathbb R\cup\{+\infty\}$ and $\omega^{-}\in\mathbb R\cup\{-\infty\}$, which are independent of $u_0$, such that
     $\omega^-+\omega^+>0$, and
     \begin{equation}\label{eq:def of ss}
     	\begin{cases}
     		\displaystyle{\lim_{t\rightarrow+\infty}}\sup \limits_{x\in (-\infty, (-\omega^--\epsilon) t]\cup[ (\omega^++\epsilon) t,+\infty) }|u(t,x)|=0\text{ for every sufficiently small }\epsilon>0,\\
     		\displaystyle{\lim_{t\rightarrow\infty}}\sup \limits_{x\in [(-\omega^-+\epsilon)t, (\omega^+-\epsilon) t]}|u(t,x)-1|=0\text{  for every sufficiently small }\epsilon>0.\\
     	\end{cases}
     \end{equation}
	then  $\omega^{\pm}$ are called the spreading speeds of \eqref{eq:main eq} with respect to $\mathcal U$ in the positive and negative directions, respectively.
	
	Here we agree that if $\omega^+ = +\infty$, then the interval $[(\omega^+ + \epsilon)t, +\infty)$ in \eqref{eq:def of ss} is understood as the empty set, and the interval $[(-\omega^- + \epsilon)t, (\omega^+ - \epsilon)t]$  is replaced by $[(-\omega^- + \epsilon)t, +\infty)$. 
	Similarly, if $\omega^- = -\infty$, then the interval $(-\infty, (-\omega^- - \epsilon)t]$ is understood as the empty set, and the interval $[(-\omega^- + \epsilon)t, (\omega^+ - \epsilon)t]$ is replaced by $(-\infty, (\omega^+ - \epsilon)t]$.
	\end{definition}
	When the coefficients are almost periodic functions, Berestycki and Nadin \cite{berestycki2012spreading,berestycki2019asymptotic} proved that the spreading speeds of  \eqref{eq:main eq} with respect to $\mathcal U=\{\phi(x)\in[0,1] \text{ with compact support}\}$ exist.
   The characterization of the spreading speeds are closely related to the generalized principal eigenvalues  of the linearized operator, which are defined as follows. Let
   $$\mathcal L\phi(x):=(a(x)\phi^{\prime}(x))^{\prime}+b(x)\phi^{\prime}(x)+c(x)\phi(x),$$ and, $\text{ for }p\in\RR$,
   \begin{equation*}\label{eq:p}
   	\begin{split}
   		&\quad L_{p}\phi(x)=e^{-px}\mathcal L(e^{p\cdot}\phi)(x)\\
   		&=(a(x)\phi^{\prime}(x))^\prime+(b(x)+2pa(x))\phi^{\prime}(x)+\big(p^2a(x)+p(b(x)+a^\prime(x))+c(x)\big)\phi(x).\\
   	\end{split}
   \end{equation*}
	\begin{definition}(\cite[Definition 2.1]{berestycki2012spreading})\label{def:gpe}
		The generalized principal eigenvalues associated with operator \(L_p\) are:
		\[
		\underline{\lambda}(p) := \sup \{\lambda \mid \exists \phi \in \mathcal{A}_{-\infty} \text{ such that } L_p\phi \geq \lambda\phi \text{ on } \RR\}, 
		\]
		\[
		\overline{\lambda}(p) := \inf \{\lambda \mid \exists \phi \in \mathcal{A}_{-\infty} \text{ such that } L_p\phi \leq \lambda\phi \text{ on } \RR\}, 
		\]
		where \(\mathcal{A}_{-\infty}\) is the set of admissible test-functions over \(\mathbb{R}\):
		\[
		\mathcal{A}_{-\infty} := \left\{ \phi \in C^2(\mathbb{R}),\ 
		\phi' / \phi \in L^\infty(\mathbb{R}),\ 
		\phi > 0 \text{ in } \mathbb{R},\ 
		\lim_{|x| \to +\infty} \frac{ \ln \phi(x)}{x} = 0 \right\}.
		\]
	\end{definition}
	
	Generalized principal eigenvalues play an important role in the study of the propagation speed for problems with exponentially decaying initial data, which will be illustrated later.
	We now recall the relevant definition and 	basic properties.
	It was proved by Berestycki and Nadin \cite{berestycki2012spreading} that $\overline{\lambda}(p)\geq\underline{\lambda}(p) \text{ for any }p\in\RR$
	if $a, b$ and $c$ are only assumed to be uniformly continuous and bounded in $x$. Moreover, they showed that if $a, b$ and $c$ are almost periodic functions, then
	$$\overline{\lambda}(p)=\underline{\lambda}(p) \text{ for any } p\in\RR.$$
	In fact, they  showed that, for any $\kappa>0$ and $p\in\mathbb R$, there exists a unique solution $u_{\kappa,p}\in C^2(\mathbb R)$ to
	\begin{equation}\label{eq:key eq}
		L_pe^{u_{\kappa,p}}(x)=\kappa u_{\kappa,p}(x)e^{u_{\kappa,p}(x)}
	\end{equation}
	and the limit
	\begin{equation}\label{eq:conv of test-func}
		\lim\limits_{\kappa\to0}\kappa u_{\kappa,p}(x) \text{ exists in } L^{\infty}(\mathbb R)
		\end{equation}
	with
	\begin{equation}\label{eq:bdd of test-func}
		-\|ap^2+(b+a^\prime)p+c\|_{L^\infty(\mathbb R)}\leq\kappa u_{\kappa,p}(x)\leq\|ap^2+(b+a^\prime)p+c\|_{L^\infty(\mathbb R)}.
	\end{equation}
	Taking $e^{u_{\kappa,p}(x)}$ as a test function in $\mathcal A_{-\infty}$, one thus deduces that $\overline{\lambda}(p)\leq	\lim\limits_{\kappa\to0}\kappa u_{\kappa,p}(x)\leq\underline{\lambda}(p)$ by the definition of $\overline{\lambda}(p)$ and $\underline{\lambda}(p)$.  
	Therefore,  $\overline{\lambda}(p)=\underline{\lambda}(p)=\lim\limits_{\kappa\to0}\kappa u_{\kappa,p}(x)$ and we thus denote it by $\lambda(p)$.
	It was known that the function $\lambda(p)$ is convex, and
	$\begin{aligned}
		\lambda(p) =O(p^2)
	\end{aligned}$
	as $p\to\infty$ (\cite{berestycki2012spreading,berestycki2019asymptotic}). 
	Then there exists $p_{\pm} > 0$, such that
	\begin{center}
		$\begin{aligned}
			\inf\limits_{p > 0}\frac{\lambda(-p)}{p} 
			= \frac{\lambda(-p_+)}{p_+}
			\text{ and }\inf\limits_{p > 0}\frac{\lambda(p)}{p} 
			= \frac{\lambda(p_-)}{p_-}.
		\end{aligned}$
	\end{center}
	Furthermore, the spreading speeds of \eqref{eq:main eq} with respect to compactly supported initial data   are $\frac{\lambda(-p_+)}{p_+}$ and $\frac{\lambda(p_-)}{p_-}$ (\cite{berestycki2012spreading,liang2021propagation}).
	
	We next specify the class of initial data 	considered in this paper.
	Assume  that $h$ is defined on $I_{R_0}=\{x | |x|\geq R_0\}$ for some $R_0\geq0$ and satisfies 
		\begin{equation}\label{eq:ass of h}
	h\in C^2(I_{R_0})	\text { and } \text{sign}(h^\prime(x))=\text{sign}(x)\ \ \text{for any } |x|\ge R_0,
	\end{equation} 
     and
	\begin{equation}\label{eq:intial data}
	\lim\limits_{x\to\pm\infty}h(x)=+\infty,\ \lim\limits_{x\to\pm\infty}h^\prime(x)=p^{\pm}_0\in[0,+\infty),
	\text{ and }\lim\limits_{x\to\pm\infty}h^{\prime\prime}(x)=0.
	\end{equation}
     Then we define
	$$\mathcal U_{p^-_0,p^+_0}:=\{u_0:\mathbb R\to[0,1] \mid h:=-\ln u_0 \text{ on } I_{R_0}, \text{ and $h$ satisfying } \eqref{eq:ass of h} \text{ and }\eqref{eq:intial data}\}.$$
	If the initial datum decays exponentially, that is, $p^{\pm}_0>0$ in \eqref{eq:ass of h},  then the spreading speeds depend also on the decay rate, see, e.g., \cite{bramson78CPAM,bramson83MAMS,Lau85JDE,Mallordy95SIAM,Ninomiya-Yanagida2019JDE,Uchiyama1978} and the references therein.
	For example, let $a(x)=c(x)=1$ and $b(x)=0$ in \eqref{eq:main eq}, and assume that $u(0,x)$ is equivalent as $x \to +\infty $ to a multiple of $ e^{-p x} $ with $ 0 < p < p^* = 1$, then \( u(t,x) \) converges to a finite shift of \( \varphi_\omega(x - \omega t) \) as \( t \to +\infty \), where  $\phi_\omega$ is the so-called traveling front with speed $\omega=\frac{p^2+1}{p}$ satisfying
	$$\phi^{\prime\prime}_\omega+\omega\phi^\prime_\omega+\phi_\omega(1-\phi_\omega)=0 \text{ in }\RR, \phi_\omega(-\infty)=1, \phi_\omega(+\infty)=0,$$
	see, e.g., \cite{Mallordy95SIAM,Uchiyama1978}. Therefore, the spreading speed $\omega^+=\frac{p^2+1}{p}$. 
	
	If the initial data decay more slowly than any exponential function,  that is, $p^{\pm}_0=0$ in \eqref{eq:intial data} (hereafter, we shall refer to such initial data as sub-exponentially decaying initial data), then it is easy to see that the spreading speed is infinity by constructing subsolutions whose initial data decay like $ e^{-p x}$ for any $p>0$ small. 
	It is therefore natural to investigate the level set \(E_\theta(t)\) in the case \(p^+_0 = 0\).
    It was proved by Hamel and Roques \cite{hamel2010fast} that, for front-like initial data, the level set 
    $$E_\theta(t)\subset u_0^{-1}\{[ e^{-(c-\epsilon)t}, e^{-(c+\epsilon)t}]\}$$ when $a=1$, $b=0$ and $c$ is constant. 
    See also Henderson’s work \cite{Henderson2016nonliearity} on the periodic case.

	The goal of this paper is to establish 	spreading properties for the general heterogeneous equation \eqref{eq:main eq} with initial data $u_0\in\mathcal U_{p_0^-,p_0^+}$. 
	We will identify a unified  scaling for both exponential and sub-exponential decaying initial data, under which a Hamilton-Jacobi  limiting equation will be derived.
	  This equation will then be utilized to estimate the expansion speed of the level set.  
	  Our approach differs from that of \cite{hamel2010fast,Henderson2016nonliearity}. Before stating the result,	we also require	the following assumption:\\
	(\textbf{A2}) $\lambda(p)>0$ for any $p\in\mathbb R.$\\
	This assumption  implies, in some sense, that the problem is of monostable type (see \cite[Definition 3.4]{nadin2015critical}). 
	Indeed, if $a, b$ and $c=\partial_sf(\cdot,0)$ are almost periodic and $u_0$ is a non-null initial datum such that $0\leq u_0 \leq 1$, then the solution $u = u(t,x)$ of \eqref{eq:main eq} converges to $1$ as $t\to+\infty$ locally in $x\in\mathbb R$  by using \cite[Theorem 1.2]{liang2021propagation}. 
	In other words, $0$ is an unstable steady state, while $1$ is a globally attractive steady state. 
	Note that this assumption is equivalent to
	 $$\lambda\left(\lim\limits_{x\to\infty}\frac{1}{x}\int_0^x\frac{b(s)}{a(s)}ds/2\right)>0$$
	 since $\lambda(p)$ is symmetric with respect to $p=\lim\limits_{x\to\infty}\frac{1}{x}\int_0^x\frac{b(s)}{a(s)}ds/2$ by \cite[Theorem 3.1]{zhou2025symmetry}, and  the positivity of \(\lambda(p)\) can be ensured  (see \cite{berestycki2012spreading,liang2021propagation,zhou2025symmetry}) by one of the following assumptions:\\ 
	1. $\inf\limits_{x\in\mathbb R}\left\{4a(x)c(x)-(a^\prime+b)^2(x)\right\}>0$,\\
	2.  $\sup\limits_{x\in\mathbb R}\left(c(x)-\frac{(b(x)-a^\prime(x))^2}{4a(x)}\right)>0$, and $a, b$ and $c$ are slowly varying.\\
	3. $\inf\limits_{x}c(x)>0$ and $\lim\limits_{x\to\infty}\frac{1}{x}\int_0^x\frac{b(s)}{a(s)}ds=0$.\\
    The main  results of this paper are as follows: 	
 	\begin{theorem}\label{the:1}
		Let (\textbf{A1})-(\textbf{A2}) and (\textbf{F1})-(\textbf{F3}) hold. Then the spreading speed of \eqref{eq:main eq} with respect to $\mathcal U_{p_0^-,p_0^+}$ exists. More precisely, \eqref{eq:def of ss} holds with
		\begin{equation}\label{eq:the speed}
			\omega^+ = 
			\begin{cases}
				\frac{\lambda(-p_+)}{p_+} 
				&\text{ if } p^+_0> p_+,\\
				\frac{\lambda(-p^+_0)}{p^+_0} 
				&\text{ if } 0 < p^+_0\leq p_+, \\
					+\infty
				&\text{ if } p^+_0=0.\\
			\end{cases}
		\end{equation}
		Furthermore,  if $p^+_0=0$, then 	
		\begin{equation}\label{eq:ss}
			\begin{cases}
				\lim\limits_{t \to +\infty}	 \sup\limits_{0\leq x\leq h^{-1}(\omega t)}	 |u(t,x) - 1|  = 0, &\forall \omega \in (0,\lambda(0)),\\
				\lim\limits_{t \to +\infty}	 \sup\limits_{x \geq h^{-1}(\omega t)}	u(t,x) = 0, &\forall \omega \in(\lambda(0),+\infty).\\
			\end{cases}
		\end{equation}
		In particular, for any $\theta\in(0,1)$ fixed and $\epsilon>0$ small, there exists $T_{\epsilon,\theta}$ large such that 
		\begin{equation}\label{eq:where is level set}
				\begin{cases}
			E^+_{\theta}(t)\subset \left\{[(\omega^+(p^+_0)-\epsilon)t,(\omega^+(p^+_0)+\epsilon)t]\right\} \text{ for any }t\geq T_{\epsilon,\theta}\ \text{ if }p_0^+>0,\\
		E^+_{\theta}(t)\subset h^{-1}\left\{[(\lambda(0)-\epsilon)t,(\lambda(0)+\epsilon)t]\right\} \text{ for any }t\geq T_{\epsilon,\theta}\ \text{ if }p_0^+=0,
			\end{cases}
		\end{equation}
		where $	E^+_{\theta}(t):=E_{\theta}(t)\cap(0,+\infty)$. Besides, similar conclusions hold for $\omega^-$ and $	E^-_{\theta}(t):=E_{\theta}(t)\cap(-\infty,0)$.
	\end{theorem}
	Note that $h^{-1}(\omega t)$ and $ h^{-1}\{[(\lambda(0)-\epsilon)t,(\lambda(0)+\epsilon)t]\}$ are well defined when $t$ is sufficiently large.
	This result shows that the positive (resp., negative) spreading speed depends only on the positive (resp., negative) decay rate of the initial data. 
	In fact, more generally, we have the following conclusion. 
		\begin{theorem}\label{the:2}
		Let (\textbf{A1})-(\textbf{A2}) and (\textbf{F1})-(\textbf{F3}) hold, and  let $u$, with  $0\leq u\leq1$,  be a solution to  
		\begin{equation}\label{equation0}
			\begin{cases}
				u_t = \partial_x(a(x)\partial_xu)+b(x)\partial_xu +  f(x,u), \ & t > 0,\ x >0,\\
				u(0,x) = u_0(x)\in[0,1], \ &x >0,
			\end{cases}
		\end{equation}
		Assume that 
			\begin{equation}\label{eq:sub super int}
			\begin{cases}
			e^{-h(x)}\leq u_0(x)\leq e^{-g(x)} \text{ on }[R_0,+\infty)\text{ for some }R_0>0,\\
			h \text{ and }g \text{ satisfy	\eqref{eq:ass of h} for $x\geq R_0$ and }\eqref{eq:intial data} \text { as }x\to+\infty \text{ with the same $p_0^+$}.
			\end{cases}
		\end{equation}
			One has\\
			(i) If $p^+_0>0$, then,  for any $\delta>0$, 
		\begin{equation}\label{eq:ss in one side}
			\begin{cases}
				\lim\limits_{t \to +\infty}	 \inf\limits_{\delta\leq x\leq \omega t}	 u(t,x)>0 , &\forall \delta>0, \omega \in (0,\omega^+),\\
				\lim\limits_{t \to +\infty}	 \sup\limits_{\delta t\leq x\leq \omega t}	 |u(t,x) - 1|  = 0, &\forall 0<\delta<\omega<\omega^+\\
				\lim\limits_{t \to +\infty}	 \sup\limits_{x \geq\omega t}	u(t,x) = 0, &\forall \omega>\omega^+,\\
			\end{cases}
		\end{equation}
		 with $\omega^+$ given in \eqref{eq:the speed}.\\
		 (ii) If $p^+_0=0$, then \eqref{eq:ss} holds with $h$ in the second equality replaced by $g$.\\		 
		 (iii) \eqref{eq:where is level set} holds with $h^{-1}\left\{[(\lambda(0)-\epsilon)t,(\lambda(0)+\epsilon)t]\right\}$ replaced by $[h^{-1}((\lambda(0)-\epsilon)t),g^{-1}((\lambda(0)+\epsilon)t)$.		
	\end{theorem}
	We have the following result for front-like initial data.
	\begin{theorem}\label{thm:front like initial data}
			Let (\textbf{A1})-(\textbf{A2}) and (\textbf{F1})-(\textbf{F3}) hold, and $u$ be a solution to \eqref{eq:main eq}.
		Assume that  $0<\inf\limits_{x\leq0}u_0(x)\leq\sup\limits_{x\leq0} u_0\leq1$ and \eqref{eq:sub super int}.
		Then all the conclusions in Theorem \ref{the:2} hold, and $\delta$ in \eqref{eq:ss in one side} is taken to be $-\infty$.
	\end{theorem}
	Hamel and Roques \cite{hamel2010fast} studied the spreading properties with such initial data in the homogeneous case. 
	In contrast, at the end of this paper, we will illustrate several typical examples of such initial data in the almost periodic setting.
	\begin{corollary}\label{cor:one side}
		Let (\textbf{A1})-(\textbf{A2}) and (\textbf{F1})-(\textbf{F3}) hold, and let $u$ be a bounded solution to \eqref{equation0}	
	  with boundary condition either\\ (i) $ u(t,0)=\gamma(t)\in[0,1]$; or\\
	  (ii) $ \mu(t)u(t,0)-\nu(t)u_x(t,0)= 0$ for any $ t\geq0,$ where $\mu\geq0$, $\nu\geq0$ and $\mu^2+\nu^2>0$.\\
	 Assume  that  $\sup\limits_{x\in\RR,u\in\RR}\frac{f(x,u)}{u}<+\infty$ and \eqref{eq:sub super int}.
		 Then all the conclusions in Theorem \ref{the:2} hold.
	\end{corollary}
	In particular, Corollary \ref{cor:one side} still holds under homogeneous Dirichlet or Neumann boundary conditions.
	
	We next consider a more general class of initial data.	
	Let $\{I_n^+\}_{n=1}^{\infty}$ be a family of closed intervals in $\mathbb{R}_+$ satisfying
	\begin{equation}\label{eq:interval_conditions}
		\text{(i) } I_n^+ \cap I_m^+ = \varnothing \text{ for } n \neq m; \qquad
		\text{(ii) } 0 < 2\kappa \le |I_n^+| \le M < \infty \;\; \forall n,
	\end{equation}
	where $|I_n^+|$ denotes the length of $I_n^+$, and $\kappa, M$ are constants. The family $\{I_n^+\}$ is said to be relatively dense in $\mathbb{R}_+$ if there exists $L>0$ such that for every $x\in\mathbb{R}_+$, there exists some $n$ satisfying
	\begin{equation}\label{eq:relatively_dense_plus}
		I_n^+ \subset [x, x+L].
	\end{equation}
	The notion of relative denseness in $\mathbb{R}_-$ can be defined analogously, with the inclusion condition replaced by $I_n^- \subset [x-L, x]$.
	Let $\mathcal{I}^+ = \bigcup_n I_n^+ \text{ and } \mathcal{I}^- = \bigcup_n I_n^-,$
	where $\{I_n^+\}$ and $\{I_n^-\}$ are relatively dense families of intervals in $\mathbb{R}_+$ and $\mathbb{R}_-$, respectively. Define
	\[
	\mathcal{I} = \mathcal{I}^+ \cup \mathcal{I}^-.
	\]
	Let $A\subset \mathbb R$. Denote 
	\[
	\mathds{1}_{A}(x) = 1 \text{ if }x\in A,\text{ and }0 \text{ if }x\in A^c.
	\]
	Now we may consider the propagation under the following more general initial data:
	\begin{theorem}\label{thm:general initial data}
		Let (\textbf{A1})-(\textbf{A2}) and (\textbf{F1})-(\textbf{F3}) hold, and  let $u$, with  $0\leq u\leq1$,  be a solution to  \eqref{equation0}.
		Assume that 
		\begin{equation}\label{eq:sub super int0}
			\begin{cases}
				\mathds{1}_{\mathcal I^+}e^{-h(x)}\leq u_0(x)\leq e^{-g(x)} \text{ on }[R_0,+\infty)\text{ for some }R_0>0,\\
				h \text{ and }g \text{ satisfy	\eqref{eq:ass of h} for $x\geq R_0$ and }\eqref{eq:intial data} \text { as }x\to+\infty \text{ with the same $p_0^+$}.
			\end{cases}
		\end{equation}
		Then all the conclusions in Theorem \ref{the:2} hold. 
		In particular, Corollary \ref{cor:one side} remains valid if \eqref{eq:sub super int} is replaced by \eqref{eq:sub super int0}
	\end{theorem}
	From this theorem, we  immediately obtain the following corollary.
	\begin{corollary}\label{cor:general initial data}
		Let (\textbf{A1})-(\textbf{A2}) and (\textbf{F1})-(\textbf{F3}) hold, and $u$ be a solution to \eqref{eq:main eq}.\\
		(i) If there exist $h, g\in\mathcal U_{p^-_0,p^+_0}$ such that
		$$\mathds{1}_{\mathcal I}	e^{-h(x)}\leq u_0(x)\leq e^{-g(x)} \text{ for }|x|>R_0,$$
		then all the conclusions in Theorem \ref{the:1} hold.\\
		(ii) If   $0<\inf\limits_{x\leq0}u_0(x)\leq\sup\limits_{x\leq0} u_0\leq1$ and 	\eqref{eq:sub super int0} is satisfied,
		then all the conclusions in  Theorem \ref{thm:front like initial data} hold.
	\end{corollary}

	\textbf{Outline of this paper:} In Section 2, we will derive a Hamilton-Jacobi equation under an appropriate scaling.
	 In Section 3, we will give an explicit expression for the solution of the Hamilton–Jacobi equation obtained in Section 2.
	  Section 4 is devoted to proving of Theorems \ref{the:1} and \ref{the:2}, and to giving four types of typical examples.

	\section{Homogenization of the equation}
	Recall the  equation we will consider is
	\begin{equation}\label{equation1}
		\begin{cases}
			u_t = \partial_x(a(x)\partial_xu)+b(x)\partial_xu +  f(x,u), \ &x \in \RR,\  t > 0,\\
			u(0,x) = u_0(x)=e^{-h(x)}\in\mathcal U_{p_0^-,p_0^+}, \ &x \in \RR.
		\end{cases}
	\end{equation}
	In this section, we proceed as follows  for equation \eqref{equation1}. First, we construct a suitable scaling that depends on the initial data.
	 Second, we determine the equation satisfied by the homogenized limit function of the solution under this scaling. 
	 It is worth noting that we only consider the case 	$x\geq R_0$; the case $x\leq R_0$ can be treated analogously.
	After a translation, we may assume without loss of generality that  $h$ satisfies:
	\begin{equation}\label{eq:h}
		h(x)\in C^2([0,+\infty)),\	h(x)\geq0 \text { and } h^\prime(x)>0\ \ \text{for any }x\geq0.
	\end{equation}	
	Furthermore, we define the following scaling transformation:
	\begin{equation}\label{eq:def of psi_varep}
				\psi_\epsilon(x):=h^{-1}\left(\frac{h(x)-h(0)}{\epsilon}+h(0)\right),\  \text{for any } x\geq0.
	\end{equation} 
    Set
    \begin{equation}\label{eq:def of u_varep}
    	u_\epsilon(t,x):=u\left(\frac{t}{\epsilon},\psi_\epsilon(x)\right),\ t\geq0, x\geq0,
    \end{equation}
    and 
	 \begin{equation}\label{eq:def of z_varep}
		Z_\epsilon(t,x):=\epsilon\ln u_\epsilon(t,x),\  t\geq0, x\geq0.
	\end{equation}
	Then one can  check that, on $\{(t,x)|t>0,x\geq0\}$, $Z_\epsilon(t,x)$ satisfies
	\begin{equation}\label{eq:z_varep}
	\partial_t Z_\epsilon =\frac{\partial_x(a(\psi_\epsilon)\partial_{x} Z_\epsilon)}{\epsilon(\psi_{\epsilon}^{\prime})^2} +\frac{a(\psi_\epsilon)}{(\epsilon\psi_\epsilon^{\prime})^2}(\partial_{x} Z_\epsilon)^2 +\left(\frac{b(\psi_\epsilon)}{\epsilon\psi_\epsilon^{\prime}}-\frac{a(\psi_{\epsilon})\psi_{\epsilon}^{\prime\prime}}{\epsilon(\psi_{\epsilon}^{\prime})^3}\right)\partial_{x} Z_\epsilon+f(\psi_\epsilon,u_\epsilon)/u_\epsilon,
	\end{equation}
	where all the derivative at $x=0$ is understood as the right-hand derivative. 
	Moreover,
	\begin{equation}\label{eq:z_varep ini-bou}
		\begin{cases}
			Z_\epsilon(0,x) =(1-\epsilon)h(0) -h(x),\ x\geq0,\\
			Z_\epsilon(t,0) = \epsilon \ln u(\frac{t}{\epsilon},0),\ t\geq0,
		\end{cases}
	\end{equation}
	Let $Z^*$ and $Z_*$ be the half-relaxed limits of $Z_\epsilon$, that is, for any $(t,x)\in\overline{\mathbb R_+^2}$,
	\begin{center}
		$\begin{aligned}
			Z^*(t,x) = \limsup_{(s,y)\in[0,+\infty)\times[0,+\infty)\atop (s,y) \to (t,x), \epsilon \to 0 }Z_\epsilon(s,y), \quad 
			Z_*(t,x) = \liminf_{(s,y)\in[0,+\infty)\times[0,+\infty)\atop (s,y) \to (t,x), \epsilon \to 0}Z_\epsilon(s,y).
		\end{aligned}$
	\end{center}
	 Then we have 
	 \begin{lemma}\label{lem:well defined and initial value}
	 	$Z^*$ and $Z_*$ are well-defined,
	 	and
	 	\begin{equation*}
	 		Z^*(0,x) = Z_*(0,x) = h(0) - h(x), \quad \forall x \geq 0.
	 	\end{equation*}
	 \end{lemma}
	 
	 \begin{proof}
	 	Define $\overline{u}(t,x) =\min\{1, \overline{c}\exp(-h(x) + \overline{c}t)\}$ and $\underline{u}(t,x) = \underline{c}^{-1}\exp(-h(x) - \underline{c}t)$,
	 	where
	 	\begin{equation*}
	 		\overline{c} \geq \max\big\{1, \sup_{x \geq 0}\{a(x)(h'(x))^2 - a(x)h''(x) - (b(x)+a^\prime(x))h'(x) + \partial_sf(x,0)\}\big\}
	 	\end{equation*}
	 	and
	 	\begin{equation*}
	 		-\underline{c} \leq \min\big\{-1, \inf_{x \geq 0}\{a(x)(h'(x))^2 - a(x)h''(x) - (b(x)+a^\prime(x))h'(x)\big\}
	 	\end{equation*}
	 	are constants to be determined later.
	 	Then we  deduce that $\overline{u}$ satisfies
	 	\begin{equation}\label{supersolution}
	 		\begin{aligned}
	 			\partial_t \overline{u}
	 			&= \overline{c}\overline{u}
	 			\geq (a(x)(h'(x))^2 - a(x)h''(x) - (b(x)+a^\prime(x))h'(x) +  \partial_sf(x,0))\overline{u} \\
	 			&= \partial_x(a(x)\partial_{x}\overline{u}) + b(x)\partial_x \overline{u} +  \partial_sf(x,0)\overline{u} \\
	 			&\geq \partial_x(a(x)\partial_{x}\overline{u}) + b(x)\partial_x\overline{u} + f(x,\overline{u})
	 		\end{aligned}
	 	\end{equation}
	 	in $(0,\infty) \times (0,\infty)$ 
	 	with $\overline{u}(0,x) = \overline{c}e^{-h(x)}$ for any $x \geq 0$,
	  and	$\underline{u}$ satisfies
	 	\begin{equation}\label{subsolution}
	 		\begin{aligned}
	 			\partial_t \underline{u}
	 			&= -\underline{c}\underline{u}
	 			\leq (a(x)(h'(x))^2 - a(x)h''(x) - (b(x)+a^\prime(x))h'(x))\underline{u} \\
	 			&\leq (a(x)(h'(x))^2 - a(x)h''(x) - (b(x)+a^\prime(x))h'(x))\underline{u} + f(x,\underline{u}) \\
	 			&= \partial_x(a(x)\partial_{x}\underline{u}) + b(x)\partial_x\underline{u} + f(x,\underline{u})
	 		\end{aligned}
	 	\end{equation}
	 	in $(0,\infty) \times (0,\infty)$ with $\underline{u}(0,x) = \underline{c}^{-1}e^{-h(x)}$ for any $x \geq 0$.
	 	Here, we have used the fact that $\underline{u} \in [0,1]$ in the second inequality of \eqref{subsolution}.
	 	
	 	Note that  $\lim\limits_{t \to \infty}u(t,0) = 1$.
	 	Therefore, we can take $\overline{c}$ and $\underline{c}$ large enough such that
	 	\begin{equation*}
	 		\underline{u}(t,0) \leq u(t,0) \leq \overline{u}(t,0), \quad \forall t \geq 0.
	 	\end{equation*}
	 	By comparison principle,
	 	we obtain
	 	\begin{equation*}
	 		\underline{u}(t,x) \leq u(t,x) \leq \overline{u}(t,x), 
	 		\quad \forall t \geq 0 \hbox{ and } x \geq 0.
	 	\end{equation*}
	 	Thus we conclude that
	 	\begin{equation*}
	 		Z_\epsilon(t,x) 
	 		= \epsilon \ln u\left(\frac{t}{\epsilon},\psi_\epsilon(x)\right)
	 		\geq \epsilon \ln \underline{u}\left(\frac{t}{\epsilon},\psi_\epsilon(x)\right) 
	 		=-\epsilon \ln \underline{c} + h(0) - h(x) - \epsilon h(0) - \underline{c}t
	 	\end{equation*}
	 	and
	 	\begin{equation*}
	 		Z_\epsilon(t,x) 
	 		= \epsilon \ln u\left(\frac{t}{\epsilon},\psi_\epsilon(x)\right)
	 		\leq \epsilon \ln \overline{u}\left(\frac{t}{\epsilon},\psi_\epsilon(x)\right) 
	 		=\epsilon \ln \overline{c} + h(0) - h(x) - \epsilon h(0) + \overline{c}t
	 	\end{equation*}
	 	for all $t \geq 0 $ and $x \geq 0$,
	 	which yields that $Z^*$ and $Z_*$ are well-defined and $Z^*(0,x) = Z_*(0,x) = h(0) - h(x)$ for all $x \geq 0$.
	 \end{proof}
	 
	 Denote 
	 $$LSC(\overline{\mathbb R_+^2}):=\{u \text{ is low semi-continuous function on }\overline{\mathbb R_+^2}\},$$
	 $$USC(\overline{\mathbb R_+^2}):=\{u \text{ is upper semi-continuous function on }\overline{\mathbb R_+^2}\}.$$
	 Then it is easy to check by the definition that $Z_*\in LSC(\overline{\mathbb R_+^2})$ and $Z^*\in USC(\overline{\mathbb R_+^2})$.	The main result of this section is
	 \begin{theorem}\label{thm:homo- eq}
	 		$Z^*$ and $Z_*$ are viscosity subsolution and supersolution to
	 	\begin{equation}\label{eq:hj eq}
	 		\begin{cases}
	 			\max\{\partial_t Z - \lambda(\frac{p^+_0}{h'(x)}\partial_x Z),Z\} = 0, \quad &t > 0, x > 0 \\
	 			Z(t,0) = 0, \quad &t > 0, \\
	 			Z(0,x) = h(0) - h(x), \quad &x \geq 0,
	 		\end{cases}
	 	\end{equation}
	 	respectively.
	  \end{theorem}
	  Before proving this theorem, we give some properties of such rescaling $\psi_\epsilon$.
	 
	 \begin{proposition}\label{prop:psi}
	 	Assume that \eqref{eq:h} holds. Then one has\\
	 	(i) The function $\psi_\epsilon(\cdot)$ is increasing on $[0,+\infty)$.
	 	 Moreover,  for any $x\geq0$, one has $\psi_{\epsilon_1}(x)<\psi_{\epsilon_2}(x)$ if $\epsilon_1>\epsilon_2$.\\
	 	(ii) $\lim\limits_{\epsilon\to0}\psi_\epsilon(x) =+\infty$  uniformly for $x\in[a,+\infty)\subset(0,+\infty)$, 
	 	and $\lim\limits_{x\to+\infty}\psi_\epsilon(x)-x=+\infty$ uniformly for $\epsilon$ small, say, $\epsilon\in(0,1/2)$.\\
	 	(iii)	 $\lim\limits_{\epsilon\to0^+}\frac{h^\prime(x)}{\epsilon\psi_{\epsilon}^{\prime}(x)}=p^+_0$  uniformly for $x\in[a,+\infty)\subset(0,+\infty)$.
	 	Moreover, if $p^+_0=0$, then
	 	$$\lim\limits_{\epsilon\to0^+}\frac{1}{\epsilon\psi_{\epsilon}^{\prime}(x)}=0$$
	 	uniformly for $ x\in[a,b]\subset[0,+\infty)$.\\
	 	(iv) $ \lim\limits_{\epsilon\to0^+}\frac{\psi_{\epsilon}^{\prime\prime}(x)}{\epsilon(\psi_{\epsilon}^{\prime}(x))^3}=0$ $\text{ uniformly for } x\in[a,b]\subset(0,+\infty)$.
	 		Moreover, if $p^+_0>0$, then this limit holds uniformly for $x\in[a,+\infty)$.
	 	
	 \end{proposition}
	 \begin{proof}
	 	(i) can be obtained  by the definition of $\psi_\epsilon$.
	 	
	 	Proof of (ii): 
	 	Since $\psi_\epsilon$ is positive and increasing on $[0,+\infty)$. Then by the monotonicity of $h$, we have 
	 	$$\psi_{\epsilon}(x)=h^{-1}\left(\frac{h(x)-h(0)}{\epsilon}+h(0)\right)\geq h^{-1}\left(\frac{h(a)-h(0)}{\epsilon}+h(0)\right)\to +\infty \text{ for any }  x\geq a.$$
	 
	 	Let us now prove that $\lim\limits_{x\to\infty}\psi_\epsilon(x)-x=+\infty$ uniformly with respect to $\epsilon\in(0,1/2)$. 
	 	In fact, since  $h(x)/2>h(0)$ for $x$ large enough, $h^{-1}$ is increasing, and $\left(h^{-1}(x)\right)^\prime\to1/p^+_0\in(0,+\infty]$ as $x\to+\infty$, we have
	 	\begin{equation*}
	 		\begin{aligned}
	 			\psi_{\epsilon}(x)-x
	 			&=h^{-1}\left(\frac{h(x)-h(0)}{\epsilon}+h(0)\right)-x\geq h^{-1}\left(2{h(x)-h(0)}\right)-h^{-1}(h(x))\\
	 			&\geq h^{-1}\left(3h(x)/2\right)-h^{-1}(h(x))=\int_{h(x)}^{3h(x)/2}\left(h^{-1}(s)\right)^\prime ds\to+\infty\ \text{ as }x\to+\infty.
	 		\end{aligned}
	 	\end{equation*}
	 	 	
	 	Proof of (iii):	 	From $h(\psi_{\epsilon}(x))=\frac{h(x)-h(0)}{\epsilon}+h(0)$, we have
	 	\begin{equation}\label{eq:diff wrt x}
	 		h^\prime(\psi_\epsilon(x))\psi_\epsilon^{\prime}(x)=\frac{h^\prime(x)}{\epsilon}
	 	\end{equation}
	 	Hence  $\lim\limits_{\epsilon\to0}\frac{h^\prime(x)}{\epsilon\psi_\epsilon^{\prime}(x)}=\lim\limits_{\epsilon\to0}h^\prime(\psi_\epsilon(x))\to p^+_0$ uniformly with respect to $x\in[a,+\infty)\subset(0,+\infty)$. 
	 	If $p^+_0=0$, then 
	 	$$\left|\frac{1}{\epsilon\psi_\epsilon^{\prime}(x)}\right|=\left|\frac{h^\prime(\psi_\epsilon(x))}{h^\prime(x)}\right|\leq\left|\frac{h^\prime(\psi_\epsilon(x))}{\inf\limits_{x\in[a,b]}h^\prime(x)}\right|\to 0$$ 
	 	as $\epsilon\to0$ uniformly with respect to $x\in[a,b]\subset[0,+\infty)$.
	 	
	 	Proof of (iv) : 	 	Differentiating both sides of \eqref{eq:diff wrt x} in $x$ yields that 
	 	$$h^{\prime\prime}(\psi_\epsilon(x))(\psi^{\prime}_\epsilon(x))^2+h^\prime(\psi_\epsilon(x))\psi_\epsilon^{\prime\prime}(x)=\frac{h^{\prime\prime}(x)}{\epsilon}.$$
	 	Therefore,  using \eqref{eq:diff wrt x}, one has 
	 	\begin{equation*}
	 		\begin{aligned}
	 			\left|\frac{\psi_\epsilon^{\prime\prime}(x)}{\epsilon(\psi^{\prime}_\epsilon(x))^3}\right|
	 			&=\left|\frac{h^{\prime\prime}(x)/h^\prime(x)}{\epsilon(\psi^{\prime}_\epsilon(x))^2}-\frac{h^{\prime\prime}(\psi_\epsilon(x))}{h^\prime(x)}	\right|\\
	 			&\leq\frac{\epsilon|h^{\prime\prime}(x)||h^\prime(\psi_\epsilon^{\prime}(x))|^2}{\inf\limits_{x\in[a,b]}|h^\prime(x)|^3}+\frac{|h^{\prime\prime}(\psi_\epsilon(x))|}{\inf\limits_{x\in[a,b]}|h^\prime(x)|}\to0\\
	 		\end{aligned}
	 	\end{equation*}
	 	as $\epsilon\to0$ uniformly with respect to $x\in[a,b]\subset[0,+\infty)$. 
	 	
	 	If $p^+_0>0$, then it is easy to see that $\inf\limits_{x\in[0,+\infty)} h^\prime(x)>0$. Therefore,
	 	\begin{equation*}
	 		\begin{aligned}
	 			\left|\frac{\psi_\epsilon^{\prime\prime}(x)}{\epsilon(\psi^{\prime}_\epsilon(x))^3}\right|
	 			\leq\frac{\sup\limits_{x\in[a,+\infty)}\epsilon|h^{\prime\prime}(x)||h^\prime(\psi_\epsilon^{\prime}(x))|^2}{\inf\limits_{x\in[a,+\infty)}(h^\prime(x))^3}+\frac{|h^{\prime\prime}(\psi_\epsilon(x))|}{\inf\limits_{x\in[a,+\infty)}h^\prime(x)}\to0,
	 		\end{aligned}
	 	\end{equation*}
	 	which yields that  the limit in (iv) holds uniformly with respect to $x\in(a,+\infty)$.
	 \end{proof}
	 
	  Let us now prove Theorem \ref{thm:homo- eq}. 
	 \begin{proof}[Proof of Theorem  \ref{thm:homo- eq}]
	 	The  initial condition follows from Lemma \ref{lem:well defined and initial value}. The proof is divided into two steps:
	 	
	 	Step 1: Prove that $Z_*$ and $ Z^*$ satisfy the boundary condition.\\
	 	If $p^+_0 > 0$,
	 	then there exists $m > 0$ such that $h' > m$ for all $x > 0$.
	 	Thus we have 
	 	$(h^{-1})'(x) = (h'(h^{-1}(x)))^{-1} < m^{-1}$ 
	 	for all $x > h(0)$,
	 	which implies that
	 	\begin{equation*}
	 		\psi_\epsilon(x) < \frac{h(x) - h(0)}{m} \cdot \frac{1}{\epsilon} \text{ for all $x > 0$.}
	 	\end{equation*}
	 	On the other hand, taking a solution $\tilde u$ of \eqref{eq:main eq} with compact supported initial data smaller than $e^{-h(x)}$, then by \cite[Theorem 1.2]{liang2021propagation} and \textbf{(A2)},  	there exists $c_0 > 0$ such that
	 	\begin{equation*}
	 	1=\lim\limits_{t \to +\infty} \inf_{|x| \leq c_0t} \tilde u(t,x) \leq\lim\limits_{t \to +\infty} \inf_{|x| \leq c_0t} u(t,x) \leq 1.
	 	\end{equation*}
	 	Fixed $t_0 > 0$,
	 	for any sequence $(s_n,y_n) \in (0,+\infty) \times (0,+\infty) \to (t_0,0)$ as $n \to \infty$, 
	 	one has 
	 	\begin{equation*}
	 		\epsilon \psi_\epsilon(y_n) < \frac{h(y_n) - h(0)}{m}  < c_0s_n
	 	\end{equation*}
	 	when $n$ is large.
	 	This yields that
	 	\begin{equation*}
	 		0 \geq Z_\epsilon(s_n,y_n) 
	 		= \epsilon \ln u (\frac{s_n}{\epsilon}, \frac{\epsilon\psi_\epsilon(y_n)}{\epsilon})
	 		\geq \epsilon \ln \inf_{|x| \leq c_0s_n/\epsilon}  u(s_n/\epsilon,x)
	 		\to 0 
	 		\quad \hbox{ as } \epsilon \to 0^+.  
	 	\end{equation*}
	 	Thus $Z^*(t,0) = Z_*(t,0) = 0$.
	 	
	 	Next, we prove the result for the case $p^+_0=0$.
	 	Note that $\inf\limits_{x}\partial_sf(x,0)> 0$. 
	 	Then  the principal eigenfunction $\phi_1$ to $(a\varphi')' + b\varphi' + (\partial_sf(x,0)/2)\varphi$ satisfies
	 	\begin{equation}\label{eq:ef on 01}
	 		\begin{cases}
	 			(a\varphi_L')' + b\varphi_L' + (\partial_sf(x,0)/2) \varphi_L \geq 0, \quad 0 < x < L \\
	 			\varphi_L(0) = \varphi_L(L) = 0, \\
	 			\varphi_L(x) > 0 \quad \text{ for } 0 < x < L,
	 		\end{cases}
	 	\end{equation}
	 for sufficiently large $L$.	We also assume that $\varphi_L$ attains its local maximum at $x = x^*$ with value $\varphi_L(x^*) = 1$ by multiplying a constant. 
	 A direct computation shows that
	 	\begin{equation*}
	 		u_1(x) =
	 		\begin{cases}
	 			0, \quad &x \leq 0, \\
	 			\eta\varphi_L(x), \quad &0 \leq x \leq x^*, \\
	 			\eta, \quad &x \geq x^*,
	 		\end{cases}
	 	\end{equation*}
	 	is a subsolution for $\eta > 0$ sufficiently small,  	which will be determined later.

	  Now take 
	 	\begin{equation*}
	 		u_2(t,x) = e^{\delta t - h(x)} - M e^{2(\delta t - h(x))}, 
	 	\end{equation*}
	 	where $\delta$ and $M$ will be determined later.
	 	Then $u_2 > 0$ if and only if $h(x) > \delta t + \ln M$.
	 	Furthermore, $u_2 \leq (4M)^{-1}$, and	one has
	 	\begin{equation*}
	 		\begin{aligned}
	 			\quad&\partial_t u_2 - a\partial_{xx}u_2 - b \partial_x u_2 - f(x,u_2)\\
	 			=&(\delta - a(h')^2 + ah'' + (b+a^\prime)h' - f(x,u_2)/u_2)u_2- M(\delta - 3a(h')^2 + ah'' + (b+a^\prime)h')e^{2\delta t - 2h(x)}.
	 		\end{aligned}
	 	\end{equation*}
	 	Since $h' \to 0$ and $h'' \to 0$ as $x \to \infty$,
	 	there exists $R > 0$ such that for any $x > R$,
	 	\begin{equation*}
	 		a(x)h''(x) + (b(x)+a^\prime(x))h'(x) < \delta 
	 	 \hbox{ and } 
	 		3a(x)(h'(x))^2 - a(x)h''(x) - (b(x)+a^\prime(x))h'(x) < \delta.
	 	\end{equation*}
	 	Take $M$ large enough such that $M > \max\{e^{h(R)},e^{h(1)}\}>1$, then in the region where $u_2 > 0$, i.e., $\Omega:=\{(x,t): h(x) > \delta t + \ln M\}$, one has $x>\max\{R, 1\}>x^*$.
	 	On the other hand, there exists $s_1>0$ small, such that $\inf\limits_{s\in(0,s_1], x\in\RR}\frac{f(x,s)}{s}>0$ by \eqref{eq:f is c alpha}. 
	 	Therefore, enlarging $M$ if necessary, one has $\inf\limits_{\Omega}\frac{f(x,u_2(t,x))}{u_2(t,x)}>0$.
	 	To sum up, taking $M>0$ large and $\delta> 0$ such that $2\delta <\inf\limits_{\Omega}\frac{f(x,u_2(t,x))}{u_2(t,x)}$,  one has	 	
	 	\begin{equation*}
	 		\begin{aligned}
	 			\partial_t u_2 - \partial_x(a\partial_{x}u_2) - b \partial_x u_2 - f(x,u_2)
	 			&\leq (\delta + ah'' + (b+a^\prime)h' -  f(x,u_2)/u_2)u_2 \\
	 			&\leq  (2\delta - f(x,u_2)/u_2) u_2 \leq 0,
	 		\end{aligned}
	 	\end{equation*} 
	 	which implies that $u_2$ is a subsolution whenever $u_2 > 0$.
	 	Now take $\eta < (8M)^{-1}$ small and let
	 	\begin{equation*}
	 		M_\eta = \min\{s \ | \ e^{-s} - Me^{-2s} = \eta\}(>M>e^{h(0)}).
	 	\end{equation*}
	 	Define
	 	\begin{equation*}
	 		\underline{u}(t,x) = 
	 		\begin{cases}
	 			u_1(x), \quad &h(x) \leq \delta t + \ln M_\eta, \\
	 			u_2(t,x), \quad &h(x) \geq \delta t + \ln M_\eta.
	 		\end{cases}
	 	\end{equation*}
	 Note that $\underline{u}(0,x) \leq u_0(x)$ for $x\geq0$ and $\underline{u}(t,0)=0$ for $t>0$. Then $\underline{u}(t,x)$ is a subsolution of \eqref{equation1} on $\{x\geq0, t>0\}$  (see, for example, \cite{berestycki-lions sub-supersolu}). 
	 By comparison principle,
	 	we have
	 	\begin{equation*}
	 		\underline{u}(t,x) \leq u(t,x) \quad 
	 		\hbox{ for all } (t,x) \in (0,\infty) \times (0,\infty).
	 	\end{equation*}
	 	In particular,
	 	for all $x^* \leq x \leq h^{-1}(\delta t + \ln M_\eta)$,
	 	one has $u(t,x) \geq \eta$.
	 	Combining with the fact $\lim\limits_{t \to +\infty}u(t,0) = 1$ and Harnack's inequality,
	 	we conclude that there exists $\eta_0 > 0$ such that for all $0 \leq x \leq h^{-1}(\delta t + \ln M_\eta)$ with $t$ large enough,
	 	one has $u(t,x) \geq \eta_0$.
	 	Now for any fixed $t_0 > 0$ and any sequence $(s_n,y_n) \in (0,+\infty) \times (0,+\infty)$ converging to $ (t_0,0)$,
	 	one has
	 	\begin{equation*}
	 		\frac{h(y_n) - h(0)}{s_n} < \delta
	 	\end{equation*}
	 	when $n$ is large.
	 	This yields that
	 	\begin{equation*}
	 		\begin{aligned}
	 			0 
	 			\geq Z_\epsilon(s_n,y_n) 
	 			&= \epsilon \ln u\left(\frac{s_n}{\epsilon},h^{-1}\Big(\frac{h(y_n) - h(0)}{s_n} \cdot \frac{s_n}{\epsilon} + h(0)\Big)\right) \\
	 			&\geq \epsilon \ln \inf\limits_{0 \leq x \leq h^{-1}(\delta s_n/\epsilon + \ln M_\eta)} u\left(\frac{s_n}{\epsilon},x\right) 
	 			\geq \epsilon \ln \eta_0 \to 0
	 			\quad \hbox{ as } \epsilon \to 0^+.
	 		\end{aligned}
	 	\end{equation*}
	 	Thus $Z^*(t,0) = Z_*(t,0) = 0$.
	
	Step 2: Prove that $Z^*$ is a viscosity subsolution. The other case follows from a similar argument.\\	 	
	 	Fix a test function $\phi$ and assume that $Z^* - \phi$ admits a strict maximum at $(t_0,x_0) \in (0,\infty) \times (0,\infty)$ over the ball 
	 	$$\overline{B_r} := \{(t,x) \in (0,\infty) \times (0,\infty) : |t - t_0| + |x - x_0| \leq r\},$$
	 	with $Z^*(t_0,x_0) < 0$.
	 	Denote $p = \frac{p^+_0}{h'(x_0)}\partial_x \phi (t_0,x_0)$.
	 	For any $\delta > 0$,
	 	let $\varphi \in \mathcal A_{-\infty}$ such that $\Lcal_p \varphi \leq (\lambda(p) + \delta) \varphi$.
	 	Set $w_\epsilon(x) = \epsilon \ln \varphi(\psi_\epsilon(x))$,
	 	then $w_\epsilon$ satisfies
	 	\begin{equation}\label{delta error}
	 		\frac{(a(\psi_\epsilon)w_\epsilon')'}{\epsilon (\psi_\epsilon')^2}+a^\prime(\psi_\epsilon)p
	 		- \frac{a(\psi_\epsilon)\psi_\epsilon''}{\epsilon (\psi_\epsilon')^3}w_\epsilon'
	 		+ a(\psi_\epsilon)(\frac{w_\epsilon'}{\epsilon \psi_\epsilon'} + p)^2
	 		+ b(\psi_\epsilon)(\frac{w_\epsilon'}{\epsilon \psi_\epsilon'} + p)
	 		+ \partial_sf(\psi_\epsilon,0) \leq \lambda(p) + \delta.
	 	\end{equation}
	 	
	 	The definition of $Z^*$ yields that there exist a sequence of positive numbers $\{\epsilon_n\}_n$ and a sequence $\{(s_n,y_n)\}_n$ in $\overline{B_r}$ such that 
	 	$\epsilon_n \to 0$,
	 	$s_n \to t_0$,
	 	$y_n \to x_0$
	 	and $Z_{\epsilon_n}(s_n,y_n) \to Z^*(t_0,x_0)$ as $n \to \infty$.
	 	For all $n$, let $(t_n,x_n)$ in $\overline{B_r}$ such that the function
	 	\begin{equation*}
	 		Z_{\epsilon_n} - \phi - w_{\epsilon_n} \hbox{ reaches a maximum at } (t_n,x_n) \hbox{ over } \overline{B_r}.
	 	\end{equation*} 
	 	Since the sequence $\{(t_n,x_n)\}_n$ lies in $\overline{B_r}$,
	 	we may assume,  	up to extraction,
	 	that it converges in $\overline{B_r}$ to some point,
	 	call $(T_0,X_0)$.
	 	For all $n$ and for all $(t,x) \in \overline{B_r}$,
	 	one has
	 	\begin{equation*}
	 		Z_{\epsilon_n}(t,x) - \phi(t,x) - w_{\epsilon_n}
	 		\leq Z_{\epsilon_n}(t_n,x_n) - \phi(t_n,x_n) - w_{\epsilon_n}(x_n).
	 	\end{equation*}
	 	Taking $t = s_n, x = y_n$ and letting $n \to \infty$,
	 	the definition of $Z^*$ yields that
	 	\begin{equation*}
	 		Z^*(t_0,x_0) - \phi(t_0,x_0) \leq Z^*(T_0,X_0) - \phi(T_0,X_0).
	 	\end{equation*}
	 	Hence,
	 	as $Z^* - \phi$ reaches a strict local maximum at $(t_0,x_0)$ over the ball $\overline{B_r}$,
	 	one gets $(T_0,X_0) = (t_0,x_0)$.
	 	We have thus proved that
	  	\begin{equation*}
	 		\begin{cases}
	 			Z_{\epsilon_n}(t_n,x_n) \to Z^*(t_0,x_0),\\
	 			(t_n,x_n) \to (t_0,x_0) \hbox{ as } n \to +\infty, \\
	 			Z_{\epsilon_n} - \phi - w_{\epsilon_n} \hbox{ reaches a local maximum at } (t_n,x_n).
	 		\end{cases}
	 	\end{equation*}
	 	Now, for any $n \in \NN$,
	 	one has
	 	\begin{equation*}
	 		\begin{cases}
	 			\partial_t Z_{\epsilon_n}(t_n,x_n) 
	 			= \partial_t \phi(t_n,x_n),  \\
	 			\partial_x Z_{\epsilon_n}(t_n,x_n) 
	 			= \partial_x \phi(t_n,x_n) + w_{\epsilon_n}'(x_n), \\ 
	 			\partial_x^2 Z_{\epsilon_n}(t_n,x_n)
	 			\leq \partial_x^2 \phi(t_n,x_n) + w_{\epsilon_n}''(x_n).
	 		\end{cases}
	 	\end{equation*}
	 	Combining \eqref{eq:z_varep}, we get
	 	\begin{equation*}
	 		\begin{aligned}
	 			\partial_t \phi 
	 			&\leq \frac{\partial_x(a(\psi_{\epsilon_n})(\partial_x \phi + w_{\epsilon_n}'))}{\epsilon (\psi_{\epsilon_n}')^2}
	 			+ \frac{a(\psi_{\epsilon_n})}{(\epsilon_n\psi_{\epsilon_n}')^2}(\partial_x\phi + w_{\epsilon_n}')^2 \\
	 			&\quad + \left(\frac{b(\psi_{\epsilon_n})}{\epsilon_n\psi_{\epsilon_n}'} - \frac{a(\psi_{\epsilon_n})\psi_{\epsilon_n}''}{\epsilon_n(\psi_{\epsilon_n}')^3}\right)(\partial_x\phi + w_{\epsilon_n}') + f(\psi_{\epsilon_n},u_{\epsilon_n})/u_{\epsilon_n},
	 		\end{aligned}
	 	\end{equation*}
	 	at $(t_n, x_n)$.
	 	Using \eqref{delta error}, we get 
	 	\begin{equation}\label{mid_1}
	 		\begin{aligned}
	 			\partial_t \phi 
	 			&\leq \frac{a(\psi_{\epsilon_n})}{\epsilon_n(\psi_{\epsilon_n}')^2}\partial_{xx}\phi
	 			+ a(\psi_{\epsilon_n})(\frac{\partial_x\phi}{\epsilon_n\psi_{\epsilon_n}'}-p)^2
	 			+2a(\psi_{\epsilon_n})(p+\frac{w_{\epsilon_n}'}{\epsilon_n\psi_{\epsilon_n}'})(\frac{\partial_x\phi}{\epsilon_n\psi_{\epsilon_n}'} - p)+ \lambda(p) + \delta  \\
	 			&\quad + (b(\psi_{\epsilon_n})+a^\prime(\psi_{\epsilon_n}))(\frac{\partial_x\phi}{\epsilon_n\psi_{\epsilon_n}'}-p)
	 			- \frac{a(\psi_{\epsilon_n})\psi_{\epsilon_n}''}{\epsilon_n(\psi_{\epsilon_n}')^3}\partial_{x}\phi
	 			+\frac{f(\psi_{\epsilon_n},u_{\epsilon_n})}{u_{\epsilon_n}}-\partial_sf(\psi_{\epsilon_n},0)	
	 		\end{aligned}
	 	\end{equation}
	 	at $(t_n, x_n)$.
	 	Note that $x_n\in(x_0/2,3x_0/2)$ for $n$ large. 
	 	Therefore, one has
	 	\begin{equation*}
	 		\begin{cases}
	 		 \lim\limits_{n\to+\infty}\frac{1}{\epsilon_n(\psi_{\epsilon_n}^{\prime}(x_n))^2}=\lim\limits_{n\to+\infty}\left(\frac{h^{\prime}(x)}{\epsilon_n\psi_{\epsilon_n}^{\prime}(x_n)}\right)^2\frac{\epsilon_n}{h^{\prime}(x_n)}=0,\\
	 			 \lim\limits_{n\to+\infty}\frac{\partial_x\phi(t_n,x_n)}{\epsilon_n\psi_{\epsilon_n}^{\prime}(x_n)}= \lim\limits_{n\to+\infty}\frac{h^\prime(x_n)}{\epsilon_n\psi_{\epsilon_n}^{\prime}(x_n)}\frac{\partial_x\phi(t_n,x_n)}{h^\prime(x_n)}=\frac{p^+_0}{h^\prime(x_0)}\partial_x\phi(t_0,x_0)= p,
	 		\end{cases}
	 	\end{equation*} 
	 	by (iii)  of Proposition \ref{prop:psi}, and
	 	\begin{equation}\label{eq:to0}
	 		 \frac{\psi_{\epsilon_n}^{\prime\prime}(x_n)}{\epsilon_n(\psi_{\epsilon_n}^{\prime}(x_n))^3}\to0\text{ as }n\to\infty
	 	\end{equation}
	 	by  (iv) of Proposition \ref{prop:psi}. 
	 	We deduce that, 	 	as $n \to +\infty$,
	 	\eqref{mid_1} becomes
	 	\begin{equation*}
	 		\partial_t\phi(t_0,x_0) \leq \lambda(p) + \delta.
	 	\end{equation*}
	 	Letting $\delta \to 0$,
	 	one has
	 	\begin{equation*}
	 		\max\left\{\partial_t Z^* - \lambda\Big(\frac{p^+_0}{h'(x)}\partial_x Z^*\Big), Z^*\right\} \leq 0
	 	\end{equation*}
	 	in viscosity sense.
	 \end{proof}

	\begin{remark}
		Indeed, it follows from the first step of the proof that for $c_0$ sufficiently small, 
		\begin{equation}
			\begin{cases}
			Z^*(t,x) = Z_*(t,x) = 0, \text{ for }|x|\leq c_0 t \text{ if } p_0^+>0;\\
			Z^*(t,x) = Z_*(t,x) = 0, \text{ for }|x|\leq h^{-1}(c_0 t+h(0))\text{ if } p_0^+=0.
			\end{cases}
		\end{equation}
	\end{remark}

	\section{Solving the Hamilton-Jacobi equation}

	In this section we will solve the Hamilton-Jacobi equation \eqref{eq:hj eq}. Recall that $h$ satisfies \eqref{eq:h}.
    When $p^+_0>0$, we define a function $l(x)$ on $[0,+\infty)$ by 
\begin{equation}
	\begin{cases}
		h(l(x))=p_0^+x+h(0),\\
		l(0)=0.
	\end{cases}
\end{equation}
 Then it is easy to check that $l\in C^1([0,+\infty))$ and
\begin{equation}
	 l^\prime(x)=\frac{p^+_0}{h^\prime(l(x))}>0, \ \forall x>0.
\end{equation}

\begin{lemma}\label{lem:tran}
   Assume that $p^+_0>0$.	Let $\zeta_*(t,x)=Z_*(t,l(x))$ and $\zeta^*(t,x)=Z^*(t,l(x))$. Then $\zeta_*$ (resp. $\zeta^*$) is the viscosity super- (resp. sub-) solution  to
	\begin{equation}\label{H_J}
			\begin{cases}
			\max\{\partial_t \zeta - \lambda(\partial_x \zeta),\zeta\} = 0, \quad &t > 0, x > 0 \\
			\zeta(t,0) = 0, \quad &t > 0, \\
			\zeta(0,x) = -p_0^+x, \quad &x \geq 0.
		\end{cases}
	\end{equation}
\end{lemma}
\begin{proof}
It is easy to see that $\zeta_*(t,0) = 0 \text{ for } t > 0$ and $\zeta_* (0,x) = Z_*(0,l(x))=h(0)-h(l(x))= -p^+_0x$ for $x\geq0$.
We only prove that $\zeta$ is a viscosity super-solution to 	
     	\begin{equation}\label{eq:zeta be supersol}
     		\max\{\partial_t \zeta_* - \lambda(\partial_x \zeta_* ), \zeta_* \} \geq 0.
     \end{equation}
The sub-solution  case can be proved similarly.

     Assume that there exists $\phi(t,x)$ such that $\zeta_*-\phi$ reaches its strict local minimum at $(t_0,x_0)$ with $\zeta_*(t_0,x_0)<0$. Without loss of generality, we may assume that $x_0\geq0$. Let $\psi(t,l(x))=\phi(t,x).$
     Then  $Z_*-\psi$ attains a strict local minimum at $(t_0,l(x_0))$ with $Z_*(t_0,l(x_0))=\zeta_*(t_0,x_0)<0$. By Theorem \ref{thm:homo- eq}, we have
     \begin{equation}\label{eq:tran to standard HJ}
     	\partial_t\psi(t_0,l(x_0))-\lambda\left(\frac{p^+_0}{h^\prime(l(x_0))}\partial_x\psi(t_0,l(x_0))\right)\geq0.
     \end{equation}
     Note that $\partial_t\phi(t,x)=\partial_t\psi(t,l(x))$ and $\partial_x\phi(t,x)=\partial_x\psi(t,l(x))l^\prime(x)=\frac{p^+_0}{h^\prime(l(x))}\partial_x\psi(t,l(x)).$
     The equation \eqref{eq:tran to standard HJ} becomes
     $$	\partial_t\phi(t_0,x_0)-\lambda\left(\partial_x\phi(t_0,x_0)\right)\geq0,$$
     which yields that \eqref{eq:zeta be supersol} in viscosity sense.
\end{proof}

Before solving \eqref{H_J}, we introduce some properties of the generalized principal eigenvalue $\lambda$ and its  Legendre transform  for the reader's convenience. 
Note that the function $\lambda(p)$ is convex, and
$\begin{aligned}
	\lambda(p) 
	\geq a p^2 
	+ c
\end{aligned}$
for some $a,c > 0$ by (\textbf{A2}). 
Recall that  $p_{\pm} > 0$ are the constants such that
	\begin{center}
	$\begin{aligned}
	\frac{\lambda(-p_+)}{p_+}=	\inf\limits_{p > 0}\frac{\lambda(-p)}{p} 		
		\text{ and } \frac{\lambda(p_-)}{p_-}=\inf\limits_{p > 0}\frac{\lambda(p)}{p}.
		\end{aligned}$
\end{center}
The Legendre transform $\lambda^*$ of  $\lambda$ is given by 
  i.e. 
  \begin{center}
  	$\begin{aligned}
  		\lambda^*(q) = \sup\limits_{p \in \RR} \ (pq-\lambda(p)).
  	\end{aligned}$
  \end{center}
  Moreover, $\lambda^*$ is convex, superlinear, and its Legendre transform is $\lambda$, i.e., $\lambda(p)=(\lambda^*(\cdot))^*(p)$ (see, for example \cite[Theorem 2.13]{Tran2021 HJ}).
  Note that by definition of Legendre transform, we have
  the following Fenchel's inequality
\begin{equation*}
\lambda(p) + \lambda^*(q) \geq pq \qquad \text{ for all } p, q \in \RR. 
\end{equation*}
\begin{lemma}\label{lem:prop1}
	The following are valid:\\
(i)  For any $q \in \RR$, 
there exists $p \in \RR$, 
such that
	$\begin{aligned}
		\lambda^*(q)
		= pq-\lambda(p) .
	\end{aligned}$\\
(ii) For any $p \in \RR$, 
there exists $q \in \RR$ such that
$\begin{aligned}
		\lambda(p)  
		= pq-\lambda^*(q) .
	\end{aligned}$\\
(iii)  $\lambda^*(q)$ is nonincreasing and positive in $(-\infty,\frac{\lambda(-p_+)}{-p_+})$, $ \lambda^*(\frac{\lambda(-p_+)}{-p_+}) = 0$ and 
\begin{equation}\label{eq:negative near 0-}
	\lambda^*(q) < 0\text{ if } q\in(\frac{\lambda(-p_+)}{-p_+},0].
	\end{equation}
\end{lemma}
	  
\begin{proof}
proof of (i): By the above quadratic lower bound for $\lambda(p)$, one has
\begin{center}
	$\begin{aligned}
		\lambda^*(q) 
		= \sup\limits_{p \in \RR}\{pq - \lambda(p)\} 
		= \max\limits_{p \in \RR}\{pq - \lambda(p)\},
	\end{aligned}$
\end{center}
which yields (i). 

The statement (ii) follows from the observation that  $\lambda = (\lambda^*)^*$ immediately.

Proof of (iii): Note that
	\begin{center}
		$\begin{aligned}
			\lambda^*(\frac{\lambda(-p_+)}{-p_+}) 
			&=\sup\limits_{p\in\RR} \left(\frac{\lambda(-p_+)}{-p_+}p - \lambda(p)\right)\\
			&=\max\left(\sup\limits_{p\geq0} \left( \frac{\lambda(-p_+)}{-p_+}p- \lambda(p)  \right),\sup\limits_{p<0} p\left( \frac{\lambda(-p_+)}{-p_+} - \frac{\lambda(p)}{p}  \right) \right)
		\end{aligned}$
	\end{center}
    and $\frac{\lambda(-p_+)}{-p_+}p- \lambda(p)\leq-\lambda(p)<0$	if $p \geq 0$. Then $\lambda^*(\frac{\lambda(-p_+)}{-p_+}) =\sup\limits_{p<0} p\left( \frac{\lambda(-p_+)}{-p_+} - \frac{\lambda(p)}{p}  \right)=0$ by the definition of $p_+$.
    Moreover, the convexity of $\lambda^*$ yields that $\lambda^*(q)$ is nonincreasing and nonnegative in $(-\infty,\frac{\lambda(-p_+)}{-p_+})$.
    Let us now show that $\lambda^*(q)$ is positive  in $(-\infty,\frac{\lambda(-p_+)}{-p_+})$. Suppose that $\lambda^*(q)=0$ for some $q<\frac{\lambda(-p_+)}{-p_+}$. Then we must have  $\lambda^*(0)\geq0$ by the convexity of  $\lambda^*(q)$,  which is impossible since $\lambda^*(0)=\sup\limits_{p\in\RR}  - \lambda(p)<0$ by \textbf{(A2)}.   
    
    Now for any $ q\in(\frac{\lambda(-p_+)}{-p_+},0]$, we have  $\lambda^*(q) =\sup\limits_{p<0} p\left( q - \frac{\lambda(p)}{p}  \right)$ by the similar arguments to before, which implies $\lambda^*(q) <0$ directly.
\end{proof}
For any $p \in \RR$, 
define $D\lambda(p)$ by
\begin{center}
$\begin{aligned}
D\lambda(p) 
:= \{q \in \RR : 
 \lambda(p) 
 + \lambda^*(q) 
 = pq\}
\end{aligned}$
\end{center}
By the above lemma, we know that $D\lambda(p)$ is nonempty. Moreover, it is closed and in fact the (Fr\'echet) subdifferential of $\lambda$ at $p$.
We next show that $D\lambda(p)$ is monotone in $p$ in certain sense thanks to the convexity of $\lambda$.
	  
\begin{lemma}\thlabel{monotonicity}
	The following are valid:\\
	(i) Assume $p_1 < p_2$.
Then we have  
\begin{center}
$\begin{aligned}
q_1 \leq q_2 \text{  for any $q_1 \in D\lambda(p_1)$ and $q_2 \in D\lambda(p_2)$. }
\end{aligned}$
\end{center}
In particular, if $q\in D\lambda(p_1)\cap D\lambda(p_2)$, then $D\lambda(p)=\{q\}$ for any $p\in(p_1,p_2)$.\\
(ii) Assume that $q_0\in D\lambda(p)$. Then the function $q\mapsto pq-\lambda^*(q)$ is non-decreasing in $(-\infty,q_0)$ and non-increasing in $(q_0,+\infty)$.
\end{lemma}
	  
\begin{proof}
Proof of (i): We have
\begin{center}
$\begin{aligned}
(p_2 - p_1)q_2 
&= \lambda(p_2) 
 + \lambda^*(q_2) 
 - p_1q_2\\
&\geq \lambda(p_2) 
 - \lambda(p_1) \\
&\geq p_2q_1 
 - \lambda^*(q_1) 
 - \lambda(p_1) \\
&= (p_2 - p_1) q_1
\end{aligned}$
\end{center}
The equality in the above equation holds due to Lemma \ref{lem:prop1}, while the inequality holds because $\lambda$ is the  Legendre transform of $\lambda^*$. 
Dividing both sides by  $(p_2 - p_1)$  yields the desired conclusion.
If $q\in D\lambda(p_1)\cap D\lambda(p_2)$, then, for any $r\in D\lambda(p)$ with $p\in(p_1,p_2)$, we have
$r\geq q$ since  $q\in D\lambda(p_1)$ and $r\leq q$ since $q\in D\lambda(p_2)$. Hence $D\lambda(p)=\{q\}$ for any $p\in(p_1,p_2)$.
 
Proof of (ii): Note that the function $q\mapsto pq-\lambda^*(q)$ is concave. Furthermore,
$$\max\limits_{q\in\mathbb R}\{pq-\lambda^*(q)\}=(\lambda^*(\cdot))^*(p)=\lambda(p)=pq_0-\lambda^*(q_0),$$
which means that the function $q\mapsto pq-\lambda^*(q)$ reaches its maximum at $q_0$. Combining this with the concavity of the function, we obtain the conclusion.	 
\end{proof}

Define
\begin{equation*}
	\Bcal = \{\gamma : \mathbb{R}_+ \to \mathbb{R}_+ \hbox{ is absolutely continuous} \},
\end{equation*}
and for any $\gamma \in \Bcal$,
we define the first exit time $\tau(\gamma)$ by
\begin{equation*}
	\tau(\gamma) := \inf\{t > 0 : \gamma(t) = 0\}.
\end{equation*}
Then we have
\begin{theorem}\thlabel{Lax formula}
	The viscosity solution of \eqref{H_J} has the following representation formula
	\begin{equation}\label{solutionformer}
		\zeta(t,x) 
		= -\inf\limits_{\gamma \in \Bcal}
		\sup\limits_{a \in [0,t \wedge \tau(\gamma)]}
		\{\int_{0}^{a}\lambda^*(\gamma'(s))ds + 1_{\{a=t\}}p_0^+\gamma(t):
		\gamma(0) = x\}.
	\end{equation}
\end{theorem}

To prove this,
we first introduce the Dynamic Programming Principle.
\begin{lemma}
	[Dynamic Programming Principle (DPP)]
	Let $\zeta$ be defined as in \eqref{solutionformer}.
	For any $t > 0$ and $x > 0$ and $0 < r < t$,
	we have
	\begin{equation}\label{DPP}
		\zeta(t,x) = - \inf\limits_{\gamma \in \Bcal}\sup\limits_{a \in [0,  r \wedge \tau(\gamma)]} \{\int_{0}^{a}\lambda^*(\gamma'(s))ds - 1_{\{a = r\}}\zeta(t-r,\gamma(r)) \  : \ \gamma(0) = x \}.
	\end{equation}
\end{lemma}

\begin{proof}
	For any $t > 0, x > 0$ and $0 < r < t$,
	define
	\begin{equation*}
		W(t,x) = - \inf\limits_{\gamma \in \Bcal}\sup\limits_{a \in [0,  r \wedge \tau(\gamma)]} \{\int_{0}^{a}\lambda^*(\gamma'(s))ds - 1_{\{a = r\}}\zeta(t-r,\gamma(r)) \  : \ \gamma(0) = x \}.
	\end{equation*}
	For any $\epsilon > 0$,
	there exists $\gamma \in \Bcal$ with $\gamma(0) = x$ such that
	\begin{equation*}
		\sup\limits_{a \in [0,t \wedge \tau(\gamma)]}\{\int_{0}^{a}\lambda^*(\gamma'(s))ds + 1_{\{a = t\}}p^+_0\gamma(t)\} < - \zeta(t,x) + \epsilon.
	\end{equation*}
	By definition,
	we have
	\begin{equation*}
		-\zeta(t-r,\gamma(r)) \leq \sup\limits_{a \in [r, t \wedge \tau(\gamma)]} \{\int_{r}^{a}\lambda^*(\gamma'(s))ds + 1_{\{a=t\}}p\gamma(t)\}.
	\end{equation*}
	Now take $\gamma$ as a test function in \eqref{DPP} and assume the supremum is attained at $a_0$.
	If $a_0 < r$,
	then
	\begin{equation*}
		-W(t,x) 
		\leq \int_{0}^{a_0}\lambda^*(\gamma'(s))ds
		\leq \sup\limits_{a \in [0,t \wedge \tau(\gamma)]} \{\int_{0}^{a}\lambda^*(\gamma'(s))ds + 1_{\{a=t\}}p_0^+\gamma(t)\} \leq -\zeta(t,x) + \epsilon.
	\end{equation*}
	If $a_0 = r$,
	then
	\begin{equation*}
		\begin{aligned}
			-W(t,x) 
			&\leq \int_{0}^{r}\lambda^*(\gamma'(s))ds - \zeta(t-r,\gamma(r)) \\
			&\leq  \int_{0}^{r}\lambda^*(\gamma'(s))ds + \sup\limits_{a \in [r, t \wedge \tau(\gamma)]} \{\int_{r}^{a}\lambda^*(\gamma'(s))ds + 1_{\{a=t\}}p_0^+\gamma(t)\} \\
			&\leq \sup\limits_{a \in [0,t \wedge \tau(\gamma)]} \{\int_{0}^{a}\lambda^*(\gamma'(s))ds + 1_{\{a=t\}}p_0^+\gamma(t)\} \leq -\zeta(t,x) + \epsilon.
		\end{aligned}
	\end{equation*}
	In conclusion,
	we have $-W(t,x) \leq -\zeta(t,x) + \epsilon$,
	and therefore 
	\begin{equation*}
		\zeta(t,x) \leq W(t,x)
	\end{equation*}
	by letting $\epsilon \to 0$.
	
	Conversely,
	for any $\epsilon > 0$,
	let $\gamma \in \Bcal$ with $\gamma(0) = x$ satisfying
	\begin{equation*}
		\sup\limits_{a \in [0, r \wedge \tau(\gamma)]} \{\int_{0}^{a}\lambda^*(\gamma'(s))ds - 1_{\{a=r\}}\zeta(t-r,\gamma(r))\} \leq -W(t,x) + \epsilon. 
	\end{equation*}
	If $\tau(\gamma) \leq r$,
	then
	\begin{equation*}
		\begin{aligned}
			-\zeta(t,x) 
			&\leq \sup\limits_{a \in [0,t \wedge \tau(\gamma)]}\{\int_{0}^{a}\lambda^*(\gamma'(s))ds + 1_{\{a = t\}}p_0^+\gamma(t)\} 
			= \sup\limits_{a \in [0, \tau(\gamma)]}\int_{0}^{a}\lambda^*(\gamma'(s))ds \\
			&= \sup\limits_{a \in [0, r \wedge \tau(\gamma)]} \{\int_{0}^{a}\lambda^*(\gamma'(s))ds - 1_{\{a=r\}}\zeta(t-r,\gamma(r))\} 
			\leq -W(t,x) + \epsilon. 
		\end{aligned}
	\end{equation*}
	If $\tau(\gamma) > r$,
	then there exists another path $\tilde{\gamma} \in \Bcal$ with $\tilde{\gamma}(0) = \gamma(r)$ such that
	\begin{equation*}
		\sup\limits_{a \in [0, t-r \wedge \tau(\gamma)]} \{\int_{0}^{a}\lambda^*(\tilde{\gamma}'(s))ds + 1_{\{a=t-r\}}p_0^+ \tilde{\gamma}(t-r)\} \leq -\zeta(t-r,\gamma(r)) + \epsilon. 
	\end{equation*}
	Now let 
	\begin{equation*}
		\Gamma(s) = 
		\begin{cases}
			\gamma(s) \quad &0 \leq s \leq r, \\
			\tilde{\gamma}(s-r) \quad &s \geq r.
		\end{cases}
	\end{equation*}
	It is easy to verify that $\Gamma \in \Bcal$ and $\tau(\Gamma) = \tau(\tilde{\gamma}) + r$.
	Take $\Gamma$ as a test function in \eqref{solutionformer} and assume the supremum is attained at $a_0$.
	If $a_0 < r$,
	then
	\begin{equation*}
		\begin{aligned}
	&\quad	-\zeta(t,x) \leq \int_{0}^{a_0}\lambda^*(\Gamma'(s))ds\\
	&\leq \sup\limits_{a \in [0, r \wedge \tau(\Gamma)]} \{\int_{0}^{a}\lambda^*(\gamma'(s))ds - 1_{\{a=r\}}\zeta(t-r,\gamma(r))\} \leq -W(t,x) + \epsilon. 
		\end{aligned}
	\end{equation*}
	If $a_0 \geq r$,
	then
	\begin{equation*}
		\begin{aligned}
			-\zeta(t,x)
			&\leq \int_{0}^{a_0}\lambda^*(\Gamma'(s))ds + 1_{\{a_0 = t\}}p_0^+ \Gamma(t) \\
			&= \int_{0}^{r}\lambda^*(\gamma'(s))ds + \int_{r}^{a_0}\lambda^*(\tilde{\gamma}'(s))ds + 1_{\{a_0 = t\}}p_0^+ \tilde{\gamma}(t-r) \\
			&\leq \int_{0}^{r}\lambda^*(\gamma'(s))ds + \sup\limits_{a \in [0, (t-r) \wedge \tau(\gamma)]} \{\int_{0}^{a}\lambda^*(\tilde{\gamma}'(s))ds + 1_{\{a=t-r\}}p_0^+ \tilde{\gamma}(t-r)\}\\
			&\leq \int_{0}^{r}\lambda^*(\gamma'(s))ds - \zeta(t-r,\gamma(r)) + \epsilon \\
			&\leq \sup\limits_{a \in [0, r \wedge \tau(\gamma)]} \{\int_{0}^{a}\lambda^*(\gamma'(s))ds - 1_{\{a=r\}}\zeta(t-r,\gamma(r))\} + \epsilon \leq -W(t,x) + 2\epsilon. 
		\end{aligned}
	\end{equation*}
	In conclusion,
	we have $-\zeta(t,x) \leq -W(t,x) + 2 \epsilon$.
	By letting $\epsilon \to 0$,
	we finally get that $W(t,x) \leq \zeta(t,x)$.
	As a result,
	we see that $W(t,x) = \zeta(t,x)$.
\end{proof}

\begin{proof}
[\textbf{Proof of \thref{Lax formula}}]
Let $\zeta$ be defined as in \eqref{solutionformer}.
Define the upper and lower semicontinuous envelopes of $\zeta$ by
\begin{equation*}
\overline{\zeta}(t,x) := \limsup_{(s,y) \to (t,x)}\zeta(s,y),
\qquad 
\underline{\zeta}(t,x) := \liminf_{(s,y) \to (t,x)}\zeta(s,y). 
\end{equation*}
Since $\lambda$ is finite and satisfies the quadratic lower bound,
$\lambda^*$ is finite, convex and superlinear.
Since the supremum in \eqref{solutionformer} contains the choice $a = 0$, we have $\zeta \leq 0$, and hence $\overline{\zeta} \leq 0$. 

We first prove that $\overline{\zeta}$ is a viscosity supersolution of
\begin{equation}\label{HJB_equation}
\max\{\partial_t\zeta - \lambda(\partial_x\zeta), \zeta\} = 0 \quad \text{ in } (0,\infty) \times (0,\infty).
\end{equation}
Let $\phi \in C^1$ touches $\zeta^*$ from above at $(t_0,x_0) \in (0,\infty) \times (0,\infty)$.
Let $Q_\rho := B(t_0,\rho) \times B(x_0,\rho)$ compactly supported contained in $(0,\infty) \times (0,\infty)$.
Then
\begin{equation*}
\zeta(t,x) \leq \overline{\zeta}(t,x) \leq \phi(t,x) \quad \text{ in } Q_\rho. 
\end{equation*}
It remains to prove
\begin{equation*}
\partial_t\phi(t_0,x_0) \leq \lambda(\partial_x\phi(t_0,x_0)).
\end{equation*}
Assume by contradiction that for some $\eta > 0$,
\begin{equation*}
\partial_t \phi(t,x) - \lambda(\partial_x\phi(t,x)) \geq \eta \quad \text{ in } Q_\rho.
\end{equation*}
Then, by Fenchel's inequality,
\begin{equation}\label{test1}
\lambda^*(q) + \partial_t \phi(t,x) - q \partial_x\phi(t,x) \geq \eta \quad \text{ for all } q \in \RR, \ (t,x) \in Q_\rho.
\end{equation}
Choose $(t_n,x_n) \to (t_0,x_0)$ such that
\begin{equation*}
\zeta(t_n,x_n) \to \overline{\zeta}(t_0,x_0), 
\qquad
\zeta(t_n,x_n) - \phi(t_n,x_n) \to 0.
\end{equation*}
Let $0 < r < \rho$ be fixed and small.
By the DPP, there exists $\gamma_n \in \Bcal$ with $\gamma_n(0) = x_n$ such that
\begin{equation}\label{test2}
\sup_{a \in [0,r \wedge \tau(\gamma_n)]}\big\{ \int_{0}^{a} \lambda^*(\gamma_n'(s))ds - 1_{\{a = r\}}\zeta(t_n - r, \gamma_n(r))\big\}
\leq -\zeta(t_n,x_n) + 1/n.
\end{equation}
By Jensen's inequality and the superlinearity of $\lambda^*$,
there exists $r > 0$ sufficiently small such that every path leaves $Q_\rho$ or hits $x = 0$ before time $r$ has cost larger than $-\overline{\zeta}(t_0,x_0) + 1$. 
More precisely,
Since
\begin{equation*}
\begin{aligned}
a\lambda^*(\frac{1}{a}\int_{0}^{a}\gamma_n'(s)ds)
\leq \int_{0}^{a}\lambda^*(\gamma_n'(s))ds 
&\leq \sup_{a \in [0,t \wedge \tau(\gamma_n)]}\big\{ \int_{0}^{a} \lambda^*(\gamma_n'(s))ds - 1_{\{a = r\}}\zeta(t_n - r, \gamma_n(r))\big\} \\
&\leq -\zeta(t_n,x_n) + 1/n
\leq - \overline{\zeta}(t_0,x_0) + 1
\end{aligned}
\end{equation*}
for all $a \in (0, r \wedge \tau(\gamma_n))$ and large $n$,
and 
\begin{equation*}
\lambda^*(p) \geq \alpha p^2 - C
\end{equation*}
for some $\alpha >0, C > 0$,
one readily obtains that
\begin{equation*}
-\overline{\zeta}(t_0,x_0) + 1
\geq a \lambda^*(\frac{1}{a}\int_{0}^{a}\gamma_n'(s)ds)
\geq \frac{\alpha}{a}(\gamma_n(a) - x_n)^2 - aC
\geq \frac{\alpha}{r}(\gamma_n(a) - x_n)^2 - rC
\end{equation*}
for all $a \in (0, r \wedge \tau(\gamma_n))$ and large $n$,
which implies that every path that satisfies \eqref{test2} cannot leave $Q_\rho$ or hit $x = 0$ before time $r$ for large $n$, provided $r$ sufficiently small.
Thus $\tau(\gamma_n) > r$ and $(t_n-s, \gamma_n(s)) \in Q_\rho$ for $0 \leq s \leq r$.
Integrating \eqref{test1} along $s \mapsto (t_n-s,\gamma_n(s))$ gives
\begin{equation*}
\int_{0}^{r}\lambda^*(\gamma_n'(s))ds - \phi(t_n-r,\gamma_n(r)) + \phi(t_n,x_n) \geq \eta r.
\end{equation*}
Using $\zeta \leq \phi$ at $(t_n-r,\gamma_n(r))$ and \eqref{test2},
we get
\begin{equation*}
-\zeta(t_n,x_n) + 1/n \geq -\phi(t_n,x_n) +\eta r.
\end{equation*}
Letting $n \to +\infty$ yields $0 \geq \eta r$, a contradiction.
Therefore $\overline{\zeta}$ is a viscosity subsolution of \eqref{HJB_equation}.

We next prove that $\underline{\zeta}$ is a viscosity supersolution of \eqref{HJB_equation}.
Let $\psi \in C^1$ touches $\underline{\zeta}$ from below at $(t_0,x_0) \in (0,\infty) \times (0,\infty)$.
If $\underline{\zeta}(t_0,x_0) = 0$,
the supersolution inequality is immediate.
Assume that $\underline{\zeta}(t_0,x_0) < 0$.
We show that
\begin{equation*}
\partial_t \psi(t_0,x_0) \geq \lambda(\partial_x \psi(t_0,x_0)).
\end{equation*} 
Suppose not.
Then for some $\eta > 0$ there exists $q \in \RR$ such that
\begin{equation*}
\partial_t \psi(t_0,x_0) - q \partial_x \psi(t_0,x_0) + \lambda^*(q) < -2\eta.
\end{equation*}
By continuity,
the same inequality with $-\eta$ instead of $-2\eta$ holds in a small cylinder $Q_\rho$.
Choose $(t_n,x_n) \to (t_0,x_0)$ such that
\begin{equation*}
\zeta(t_n,x_n) \to \underline{\zeta}(t_0,x_0),
\qquad
\zeta(t_n,x_n) - \psi(t_n,x_n) \to 0.
\end{equation*}
For $r$ sufficiently small,
the straight path $\gamma_n(s) = x_n + qs$ stays in $Q_\rho$ for $0 \leq s \leq r$.
Since $\underline{\zeta}(t_0,x_0) < 0$, for some $\kappa_0 > 0$ and all large $n$,
we have $-\zeta(t_n,x_n) > \kappa_0$..
The DPP applied to this straight path gives
\begin{equation*}
-\zeta (t_n, x_n) \leq \max\{0, r\lambda^*(q) - \zeta(t_n-r,x_n+qr)\}.
\end{equation*}
Since the left hand side is larger than $\kappa_0$,
the terminal term must dominate;
hence
\begin{equation*}
-\zeta(t_n,x_n) \leq r \lambda^*(q) - \zeta(t_n - r, x_n + qr).
\end{equation*}
Because $\psi \leq \underline{\zeta} \leq \zeta$,
we obtain
\begin{equation*}
\zeta(t_n,x_n) - \psi(t_n,x_n) \geq \psi(t_n-r, x_n+qr) - \psi(t_n,x_n) - r\lambda^*(q).
\end{equation*}
The choice of $q$ implies that the right hand side is at least $\eta r$ for $n$ large.
Letting $n \to +\infty$ gives a contradiction.
Hence $\underline{\zeta}$ is a viscosity supersolution of \eqref{HJB_equation}.

It remains to verify the boundary and initial values.
For $t > 0$ and $(s,y) \to (t,0)$,
take the path
\begin{equation*}
\gamma(r) = y-r.
\end{equation*}
When $(s,y)$ is sufficiently close to $(t,0)$,
we have $\tau(y) = y < s$.
Thus
\begin{equation*}
0 \leq -\zeta(s,y) \leq \max\{0, y\lambda^*(-1)\} \to 0,
\end{equation*}
and consequently 
\begin{equation*}
\lim\limits_{(s,y) \to (t,0)} Z(s,y) = 0.
\end{equation*}

For the initial condition,
let $(s,y) \to (0,x)$ with $x \geq 0$.
Taking the constant path $\gamma \equiv y$ gives
\begin{equation}\label{upper_estimate}
-\zeta(s,y) \leq \max\{0, s\lambda^*(0) + py\}.
\end{equation}
On the other hand,
by Fenchel's inequality,
for any admissible path $\gamma \in \Bcal$,
if $\tau(\gamma) \geq s$,
then
\begin{equation*}
\int_{0}^{s}\lambda^*(\gamma'(r))dr + p \gamma(s) \geq py - s \lambda(-p).
\end{equation*} 
If $\tau(\gamma) < s$,
then
\begin{equation*}
\int_{0}^{\tau(\gamma)} \lambda^*(\gamma'(r))dr
\geq py - \tau(\gamma)\lambda(-p) \geq py - s \lambda(-p),
\end{equation*}
because $\lambda(-p) > 0$.
Taking the supremum in time and then the infimum in $\gamma$ yields
\begin{equation}\label{lower_estimate}
-\zeta(s,y) \geq py - s \lambda(-p).
\end{equation}
Combing \eqref{upper_estimate} and \eqref{lower_estimate} gives
\begin{equation*}
\lim\limits_{(s,y) \to (0,x)}\zeta(s,y) = -px.
\end{equation*}

The above boundary limits give the same parabolic boundary values for $\overline{\zeta}$ and $\underline{\zeta}$.
Moreover,
the constant path also gives
\begin{equation*}
0 \leq -\zeta(t,x) \leq \max\{0, t\lambda^*(0) + px\} \leq C_T(1+x),
\qquad 0 \leq t \leq T,
\end{equation*}
which is the linear growth needed for comparison.
By the comparison principle for \eqref{H_J} on $(0,T) \times (0,\infty)$, we obtain
\begin{equation*}
\overline{\zeta} \leq \underline{\zeta}.
\end{equation*}
Since always $\underline{\zeta} \leq \overline{\zeta}$,
we conclude that $\zeta = \underline{\zeta} = \overline{\zeta}$ is the viscosity solution of \eqref{H_J}.
\end{proof}

Since the Hamiltonian is independent of $t$ and $x$, the representation formula can be simplified as follows.

\begin{theorem}
	We have
	\begin{equation}\label{solution}
		\zeta(t,x) 
		= -\inf_{y \in \RR_+, \ 0< s < t}
		\left\{\max\{0, t\lambda^*(\frac{y-x}{t}) + p_0^+y\}, \max\{0, (t-s)\lambda^*(\frac{-x}{t-s})\}\right\}.
	\end{equation}
\end{theorem}

\begin{proof}
	Fix $(t,x) \in (0,\infty) \times (0,\infty)$. 
	For each $y \in \mathbb{R}_+$, 
	let us consider the path $\gamma$ as the straight line segment connecting $(0,x)$ and $(t,y)$, 
	that is
	\begin{equation*}
		\gamma(r) 
		= x + r \frac{y-x}{t}  \quad 
		\text{ for all } r \in [0,t].
	\end{equation*} 
	Then \eqref{solutionformer} gives
	\begin{equation*}
		\zeta(t,x) 
		\geq - \sup_{a \in [0,t]}(a\lambda^*(\frac{y-x}{t}) 
		+ 1_{\{a=t\}}p_0^+y) 
		= -\max\{0, t\lambda^*(\frac{y-x}{t}) + p_0^+y\} \text{ for any }y \in \mathbb{R}_+.
	\end{equation*}
	For any $s \in (0,t)$,
	consider 
	\begin{equation*}
		\tilde{\gamma}(r) = x - \frac{x}{t-s}r, 
		\quad r \in [0,t-s].
	\end{equation*}
	Then
	\begin{equation*}
		\zeta(t,x) 
		\geq -\inf_{y \in \mathbb{R}_+, \ 0< s < t}
		\left\{\max\{0, t\lambda^*(\frac{y-x}{t}) + p_0^+y\}, 
		\max\{0, (t-s)\lambda^*(\frac{-x}{t-s})\}\right\}. 
	\end{equation*}
	On the other hand, 
	if $\gamma$ is any admissible path with $\gamma(0) = x$,
	then by Jensen's inequality, 
	we have
	\begin{equation*}
		\frac{1}{t \wedge \tau(\gamma)}\int_{0}^{t \wedge \tau(\gamma)}\lambda^*(\gamma'(r))dr 
		\geq \lambda^*\left(\frac{1}{t \wedge \tau(\gamma)}\int_{0}^{t \wedge \tau(\gamma)}\gamma'(r)dr\right).
	\end{equation*}
	If $\tau(\gamma) \geq t$,
	denote $y = \gamma(t) \geq 0$. 
	Note that
	\begin{equation*}
		\int_{0}^{t}\gamma'(r)dr 
		= \gamma(t) - \gamma(0) 
		= y - x.
	\end{equation*}
	Hence
	\begin{equation*}
		t\lambda^*(\frac{y-x}{t}) + p_0^+y
		\leq \int_{0}^{t}\lambda^*(\gamma'(s))ds 
		+ p_0^+\gamma(t) 
		\leq \sup_{a \in [0,t \wedge \tau(\gamma)]}(\int_{0}^{a}\lambda^*(\gamma'(s))ds 
		+ 1_{\{a=t\}}p_0^+\gamma(t)).
	\end{equation*}
	If $\tau(\gamma) < t$,
	denote $s =t- \tau(\gamma)$.
	Note that
	\begin{equation*}
		\int_{0}^{t-s}\gamma'(r)dr = \gamma(t-s) - \gamma(0) = -x.
	\end{equation*}
	Thus
	\begin{equation*}
		(t-s)\lambda^*(\frac{-x}{t-s}) 
		\leq \int_{0}^{t-s}\lambda^*(\gamma'(s))ds
		\leq \sup_{a \in [0,t \wedge \tau(\gamma)]}(\int_{0}^{a}\lambda^*(\gamma'(s))ds 
		+ 1_{\{a=t\}}p_0^+\gamma(t)). 
	\end{equation*}
	Furthermore, note that
	\begin{equation*}
		\sup_{a \in [0,t \wedge \tau(\gamma)]}(\int_{0}^{a}\lambda^*(\gamma'(s))ds 
		+ 1_{\{a=t\}}p_0^+\gamma(t)) 
		\geq 0.
	\end{equation*}
	From this, we get
	\begin{equation*}
		\begin{aligned}
			\max\{0,A t\lambda^*(\frac{y-x}{t}) + p_0^+y\}
			&= \sup_{a \in [0, t]}\int_{0}^{a}\lambda^*(\frac{y-x}{t})dr + 1_{\{a = t\}}p_0^+y \\
			&\leq \sup_{a \in [0,t \wedge \tau(\gamma)]}(\int_{0}^{a}\lambda^*(\gamma'(s))ds 
			+ 1_{\{a=t\}}p_0^+\gamma(t))
		\end{aligned}
	\end{equation*}
	if $\tau(\gamma) \geq t$, and 
	\begin{equation*}
		\begin{aligned}
			\max\{0, (t-s)\lambda^*(\frac{-x}{t-s})\}
			&= \sup_{a \in [0, t-s]}\int_{0}^{a}\lambda^*(\frac{-x}{t-s})dr \\
			&\leq \sup_{a \in [0,t \wedge \tau(\gamma)]}(\int_{0}^{a}\lambda^*(\gamma'(s))ds 
			+ 1_{\{a=t\}}p_0^+\gamma(t))
		\end{aligned}
	\end{equation*}
	if $\tau(\gamma) < t$.
	Thus the conclusion follows.
\end{proof}

We may further reduce the solution of \eqref{H_J} to the following conclusion. 

\begin{theorem}\label{thm:zeta}
	There holds
	\begin{equation*}
		\begin{aligned}
			\zeta(t,x) =
			\begin{cases}
				-\max \{0, p_0^+x - \lambda(-p_0^+)t\}, \quad 
				&t > 0, \ x \geq -\max \{D\lambda(-p_0^+)\}t, \\
				-\max \{0, t\lambda^*(-\frac{x}{t})\}, \quad
				&t > 0, \ 0 < x < -\max \{D\lambda(-p_0^+)\}t.
			\end{cases}
		\end{aligned}
	\end{equation*}
\end{theorem}

\begin{proof}
	We claim that
	\begin{equation*}
		\zeta(t,x) = -\inf_{y \in \mathbb{R}_+} \max\{0, t\lambda^*(\frac{y-x}{t}) + p_0^+y\}.
	\end{equation*}
	In fact,
	by \eqref{solution},
	we only need to prove that
	\begin{equation*}
		\inf_{0 < s < t}\max\{0, (t-s)\lambda^*(\frac{-x}{t-s})\} = \max \{0, t\lambda^*(\frac{-x}{t})\}.
	\end{equation*}
	Note that $\lambda^*$ is convex and $\lambda^*(0) < 0$,
	thus we have
	\begin{equation*}
		\frac{t-s}{t}\lambda^*(\frac{-x}{t-s})
		> \frac{s}{t}\lambda^*(0) + \frac{t-s}{t}\lambda^*(\frac{-x}{t-s}) 
		\geq \lambda^*(\frac{-x}{t}).
	\end{equation*}
	Therefore
	\begin{equation*}
		\begin{aligned}
			\zeta(t,x) 
			&= -\inf_{y \in \mathbb{R}_+} \max\{0, t\lambda^*(\frac{y-x}{t}) + p_0^+y\} \\
			&=-\inf_{q \geq -\frac{x}{t}}\max\{0, t\lambda^*(q) + p_0^+(x+qt)\} \\
			&= -\inf_{q \geq -\frac{x}{t}}\max\{0, t(\lambda^*(q) + p_0^+q) + p_0^+x\} \\
			&= 
			\begin{cases}
				-\max\{0, p_0^+x - \lambda(-p_0^+)t\}, 
				&\hbox{ if } t > 0, \ x \geq -\max\{D\lambda(-p_0^+)\}t \\
				-\max\{0, t\lambda^*(-\frac{x}{t})\}, 
				&\hbox{ if } t > 0 , \ 0 \leq x < -\max\{D\lambda(-p_0^+)\}t.
			\end{cases}
		\end{aligned}
	\end{equation*}
	This completes the proof.
\end{proof}
    
By comparison, we have
 $$Z_*(t,l(x)) = \zeta_*(t,x)= \zeta(t,x) = \zeta^*(t,x) = Z^*(t,l(x)).$$
    From now on, we denote the viscosity solution $Z^*(t,x)=Z_*(t,x)$ of \eqref{eq:hj eq} by $Z(t,x)$, which is a continuous function. 
    Moreover, from $Z^*=Z_*$    and the definitions of these two quantities, we have the following corollary.
    
    \begin{corollary}\label{cor:locally uniformly}
    	One has
    	$$\lim\limits_{\varepsilon\to0}	Z_\varepsilon(t,x)=Z(t,x)  \text{ locally uniformly on }\overline{\mathbb R^2_+}.$$
    \end{corollary}
   
    Below, we obtain the explicit expression for the solution of \eqref{eq:hj eq} for all $p_0^+\geq0$.  
        
\begin{theorem}\label{thm:the sol}
	$Z(t,x)=Z^*(t,x)=Z_*(t,x)$ is the viscosity solution  of \eqref{eq:hj eq}.  Moreover, \\
	(i) If $p_0^+ > p_+$, 
	then
	\begin{equation*}
		\begin{aligned}
		Z(t,l(x))=\zeta(t,x) = 
			\begin{cases}
				\lambda(-p_0^+)t - p_0^+x, \quad 
				&t > 0, \ x > -\max\{D\lambda(-p_0^+)\}t; \\
				-t\lambda^*(-\frac{x}{t}), \quad 
				&t > 0, \ \lambda(-p_+)t/p_+ < x \leq -\max\{D\lambda(-p_0^+)\}t; \\
				0,  \quad 
				&t > 0, \ 0 < x \leq \lambda(-p_+)t/p_+.
			\end{cases}
		\end{aligned}
	\end{equation*}
(ii) If $0 < p_0^+ \leq p_+$, 
then
\begin{equation*}
	\begin{aligned}
		Z(t,l(x))=\zeta(t,x) = 
		\begin{cases}
			\lambda(-p_0^+)t - p_0^+x, \quad
			&t > 0, \ x > \lambda(-p_0^+)t/p_0^+; \\
			0,  \quad 
			&t > 0, \ 0 < x \leq \lambda(-p_0^+)t/p_0^+.
		\end{cases}
	\end{aligned}
\end{equation*}
(iii) If $p_0^+=0$, then
\begin{equation*}
	Z(t,x)=\min\{h(0)-h(x)+\lambda(0)t,0\}\ \forall t\geq0, x\geq0.
\end{equation*}
\end{theorem}
\begin{proof}
    The statements (i) and (ii) follow from Theorem \ref{thm:zeta} immediately.		
    One can check that the function $Z$ in (iii) is a viscosity solution  of \eqref{eq:hj eq} with $p_0^+=0$. 
    Then by  comparison principle, 	$Z^*=Z_*=Z$ is the viscosity solution.
\end{proof}

	  \section{The spreading speeds for slowly decaying initial data}
	  
	In this section, we focus primarily on the speed in the positive direction. The conclusions for the spreading speed in the negative direction can be derived analogously.
	Let $u$ be the solution to \eqref{eq:main eq} with initial data $u_0\in\mathcal U_{p_0^-,p_0^+}$.
	
	\subsection{The spreading speeds with respect to a subclass of $\mathcal U_{p_0^-,p_0^+}$}
	In this section, we will prove that Theorem \ref{the:1} holds when the initial datum $u_0(x)\in\mathcal U_{p_0^-,p_0^+}$, where $u_0=e^{-h(x)}$ with  $h$ satisfying \eqref{eq:intial data} and \eqref{eq:h}.
	Therefore, we can directly adopt the scaling $\psi_\epsilon$ given by \eqref{eq:def of psi_varep}.
	Recall that $	u_\epsilon(t,x)=u\left(\frac{t}{\epsilon},\psi_\epsilon(x)\right) \text{ and }Z_\epsilon(t,x):=\epsilon\ln u_\epsilon(t,x)$ for $t\geq0$ and $x\geq0$.
	One has
	\begin{lemma}\label{lem:er fen}
		The following are valid:\\
		(i) $\liminf\limits_{\epsilon\to0}u_\epsilon(t,x)>0 \text{ locally uniformly in }\text{int}\{(t,x)|Z(t,x)=0, t>0, x>0\}.$\\
		(ii) $\limsup\limits_{\epsilon\to0}u_\epsilon(t,x)=0 \text{ locally uniformly in } \{(t,x)|Z(t,x)<0, t>0, x>0\}.$
	\end{lemma}
	\begin{proof}
		Proof of (i): We only need to show that for any $(\tau,\xi)\in\text{int}\{(t,x)|Z(t,x)=0\}$, there exists $r>0$ such that
		$$\liminf\limits_{\epsilon\to0}u_\epsilon(t,x)>0 \text{ uniformly for }(t,x)\in B_r(\tau,\xi).$$
		Fix $(\tau,\xi)\in\text{int}\{(t,x)|Z(t,x)=0\}$. It follows from Theorem \ref{thm:homo- eq} that $Z_\epsilon=\epsilon\ln u_\epsilon(t,x)\to0$ on some neighborhood, say, $B_{2r}(t_0,x_0)$ for some $r$, uniformly.
		Now set $$\phi(t,x)=-({|t-t_{0}|}^{2}+{|x-x_{0}|}^{2}) \ \text{ for }(t_0,x_0)\in B_r(\tau,\xi).$$ 
		Then $Z_{\epsilon}(t,x)-\phi(t,x)$ reaches
		its minimum at some point, say $(t_{\epsilon},x_{\epsilon})$ (depends on $(t_0,x_0)$), over $B_{2r}(\tau,\xi)$. Moreover, it is easy to see that $\lim\limits_{\ep\to0}(t_{\epsilon},x_{\epsilon})=(t_{0},x_{0})$ uniformly with respect to $(t_0,x_0)\in B_r(\tau,\xi)$ since
		$Z_{\epsilon}(t,x)\to0$ as $\epsilon\to0$ uniformly in $B_{2r}(\tau,\xi)$.
		Without loss of generality, we may assume that $x_\epsilon\geq0$.
		Hence 	at $(t_{\epsilon},x_{\epsilon})$, we have
		$$\partial_tZ_{\epsilon}(t,x)={\partial}_{t}\phi(t,x),\ \partial_xZ_{\epsilon}(t,x)={\partial}_{x}\phi(t,x), \text{ and }
		\partial_{xx}Z_{\epsilon}(t,x)\geq\partial_{xx}\phi(t,x).$$
		 Combining these with \eqref{eq:z_varep}, we have
			\begin{equation}
				\partial_t \phi_\epsilon \geq\frac{\partial_x(a(\psi_\epsilon)\partial_{x} \phi_\epsilon)}{\epsilon(\psi_{\epsilon}^{\prime})^2} +\frac{a(\psi_\epsilon)}{(\epsilon\psi_{\epsilon}^{\prime})^2}(\partial_{x} \phi_\epsilon)^2 +\left(\frac{b(\psi_\epsilon)}{\epsilon\psi_{\epsilon}^{\prime}}-\frac{a(\psi_{\epsilon})\psi_{\epsilon}^{\prime\prime}}{\epsilon(\psi_{\epsilon}^{\prime})^3}\right)\partial_{x} \phi_\epsilon+\frac{f(\psi_\epsilon,u_\epsilon)}{u_\epsilon}
		\end{equation}
		at $(t_\epsilon,x_\epsilon)$. Using (iii) and (iv) of Proposition \ref{prop:psi} and setting $\epsilon\to0$, we have
		$$\limsup\limits_{\epsilon\to0}\frac{f(\psi_\epsilon(x_\ep),u_\epsilon(t_\ep,x_\ep))}{u_\epsilon(t_\ep,x_\ep)}\leq0,$$ which yields that
	$$\liminf\limits_{\epsilon\to0}Cu^\alpha_\epsilon(t_\ep,x_\ep)\geq\liminf\limits_{\epsilon\to0}\partial_sf(\psi_\epsilon(x_\ep),0)>0$$
	by \eqref{eq:f is c alpha}.
	Furthermore, by the definition of $(t_{\epsilon},x_{\epsilon})$, we conclude
		$$Z_{\epsilon}(t_{0},x_{0})=Z_{\epsilon}(t_{0},x_{0})-\phi(t_{0},x_{0})\geq Z_{\epsilon}(t_{\epsilon},x_{\epsilon})-\phi(t_{\epsilon},x_{\epsilon})\geq Z_{\epsilon}(t_{\epsilon},x_{\epsilon}).$$
		Thus
	    \begin{equation*}
			\begin{split}
				\liminf\limits_{\epsilon\to0}u_{\epsilon}(t_{0},x_{0})
				&\geq\liminf\limits_{\epsilon\to0}\exp\{Z_\ep(t_0,x_0)/\ep\}\geq\liminf\limits_{\epsilon\to0}\exp\{Z_\ep(t_\ep,x_\ep)/\ep\}\\
				& \geq		\liminf\limits_{\epsilon\to0}u_{\epsilon}(t_{\epsilon},x_{\epsilon})>0.
			\end{split}
		\end{equation*}
		
		Proof of (ii): Suppose that the statement fails. Then there exists $(\ep_n,t_n,x_n)$ with  $\ep_n\to0$ and 
		$$(t_n,x_n)\to(t_0,x_0)\in \overline{B}\subset\{(t,x)|Z(t,x)<0\}$$ for some ball $B$ such that  $\liminf\limits_{n\to\infty}u_{\epsilon_n}(t_n,x_n)\in(0,1]$.
		Therefore, $Z(t_0,x_0)=\lim\limits_{n\to\infty}Z_{\ep_n}(t_n,x_n)=\lim\limits_{n\to\infty}\ep_n\ln u_{\ep_n}(t_n,x_n)=0$, which is impossible since $(t_0,x_0)\in\{(t,x)|Z(t,x)<0\}$.
			\end{proof}
	With the above  lemma at hand, we have
		\begin{theorem}\label{the:1 simple case}
		Let (\textbf{A1})-(\textbf{A2}) and (\textbf{F1})-(\textbf{F3}) hold. 
		Assume that the initial datum $u_0(x)\in\mathcal U_{p_0^-,p_0^+}$, where $u_0=e^{-h(x)}$ with $h$ satisfying \eqref{eq:intial data} and \eqref{eq:h}.
		Then  
		\begin{equation*}
			\begin{cases}
				{\lim_{t\rightarrow+\infty}}\sup \limits_{x\geq\omega t}|u(t,x)|=0, \text{ for any }\omega>\omega^+,\\
				\displaystyle{\lim_{t\rightarrow\infty}}\sup \limits_{0\leq x\leq \omega t}|u(t,x)-1|=0, \text{ for any }\omega<\omega^+,
			\end{cases}
		\end{equation*}
		where $\omega^+$ is given by \eqref{eq:the speed}.
		Furthermore,  \eqref{eq:ss} and \eqref{eq:where is level set} 	hold.
	\end{theorem}
	\begin{proof}
		The proof is divided into three steps:
		
		 \textbf{Step 1:} Consider the spreading speed when   $p^+_0>0$.  
		 We only prove the case $p^+_0>p_+$; the case $0< p^+_0\leq p_+$ can be proved similarly.
		 For any $\omega\in\left(0,\frac{\lambda(-p_+)}{p_+}\right)$, we take $\kappa\in(0,\min\{\frac{\omega}{2}, \frac{\lambda(-p_+)}{2p_+}-\frac{\omega}{2}\})$ small.
		Since $\lim\limits_{x\to\infty}\frac{h(x)}{x}=p^+_0>0$,  then there exists $\epsilon_0$ small such that 
		$$\frac{p^+_0(\omega-\kappa/2)}{\epsilon}<h\left(\frac{\omega }{\epsilon}\right)<\frac{p^+_0(\omega+\kappa)}{\epsilon}, \ \forall \epsilon\leq \epsilon_0.$$
		Smaller $\epsilon_0$ if necessary, we have
		$$\frac{p^+_0(\omega-\kappa)}{\epsilon}+h(0)<h\left(\frac{\omega }{\epsilon}\right)<\frac{p^+_0(\omega+\kappa)}{\epsilon}+h(0),  \ \forall \epsilon\leq \epsilon_0,$$
		which yields that 
		\begin{equation}\label{eq:in the cone}
			\frac{\omega }{\epsilon}\in\left\{h^{-1}\left(\frac{p^+_0\tilde\omega t}{\epsilon}+h(0) \right)| \tilde\omega\in[\omega-\kappa,\omega+\kappa]\right\},\ \ \forall \epsilon\leq \epsilon_0.
		\end{equation}
        Note that $[\omega-\kappa,\omega+\kappa]\in(0,\frac{\lambda(-p_+)}{p_+})$. 
        Then         
        $$S:=\{(1,\tilde\omega )|\tilde\omega\in[\omega-\kappa,\omega+\kappa]\}\subset\text{int}\{(\tau,\xi)|\zeta(\tau,\xi)=0\}, $$
        by (i) of Theorem \ref{thm:the sol},
        which yields that
        $\{(1,l(\tilde\omega))|\tilde\omega\in[\omega-\kappa,\omega+\kappa]\}\subset\text{int}\{(\tau,\xi)|Z(\tau,\xi)=0\}$. 
        Then
        \begin{equation*}\label{}
        	\begin{split}
        	&\ \quad\liminf\limits_{\epsilon\to0}\inf\limits_{\tilde\omega\in[\omega-\kappa,\omega+\kappa]}u\left(\frac{1}{\epsilon},h^{-1}\left(\frac{p^+_0\tilde\omega }{\epsilon}+h(0) \right)\right)\\
        	&=\liminf\limits_{\epsilon\to0}\inf\limits_{\tilde\omega\in[\omega-\kappa,\omega+\kappa]}u\left(\frac{1}{\epsilon},h^{-1}\left(\frac{h(l(\tilde\omega) )-h(0)}{\epsilon}+h(0) \right)\right)\\
        	&=\liminf\limits_{\epsilon\to0}\inf\limits_{\tilde\omega\in[\omega-\kappa,\omega+\kappa]}u_\epsilon\left(1,l(\tilde\omega ) \right)>0
        	\end{split}
        \end{equation*}
       by (i) of Lemma \ref{lem:er fen}. We thus have
       \begin{equation}\label{eq:(1,w)positive}
       		\liminf\limits_{\epsilon\to0}u\left(\frac{1}{\epsilon},\frac{\omega}{\epsilon}\right)
       		\geq\liminf\limits_{\epsilon\to0}\inf\limits_{\tilde\omega\in[\omega-\kappa,\omega+\kappa]}u\left(\frac{1}{\epsilon},h^{-1}\left(\frac{p^+_0\tilde\omega }{\epsilon}+h(0) \right)\right)>0,
       \end{equation}
		where the first inequality follows from \eqref{eq:in the cone}. Note also that the solution $u = u(t,x)$ of \eqref{eq:main eq} converges to $1$ as $t\to+\infty$ locally in $x\in\mathbb R$. Combining this with \eqref{eq:(1,w)positive}, we deduce, by similar arguments to \cite[Proof of part 2 of Theorem 2.1]{berestycki2012spreading} (see also \cite[Theorems 1.3 and 1.6]{berestycki2008Asymptotic}), that 
		$$\lim\limits_{t \to +\infty}\sup\limits_{0 \leq x < \omega t}|u(t,x) - 1|  = 0\text{ for any }\omega\in(0,\frac{\lambda(-p_+)}{p_+}).$$
		The statement
		$$\lim\limits_{t \to +\infty}		\sup\limits_{x \geq \omega t}		|u(t,x)|  = 0 \text{ for any }\omega>\frac{\lambda(-p_+)}{p_+}$$
		 can be proved by using (ii) of Lemma \ref{lem:er fen} or by constructing a supersolution as was done in \cite[Proof of part 1 of Theorem 2.1]{berestycki2012spreading}. Therefore, the spreading speed in the positive direction $\omega^+=\frac{\lambda(-p_+)}{p_+}$  if $p^+_0>p_+$.
		 
		 When  $0< p^+_0\leq p_+$, all the arguments above are still valid with $=\frac{\lambda(-p_+)}{p_+}$ replaced by $=\frac{\lambda(-p_0^+)}{p^+_0}$.
		 Therefore, the spreading speed in the positive direction $\omega^+=\frac{\lambda(-p_0^+)}{p_0^+}$  if $0<p_0^+\leq p_+$.
		 
	 Consequently, the first two expressions in \eqref{eq:the speed} together with the first expression in \eqref{eq:where is level set}  are valid.
		 
		  \textbf{Step 2:} Consider the spreading speed when   $p^+_0=0$. Note that $\omega^+=+\infty$ for $p^+_0=0$ follows directly from the conclusion of case  $0< p^+_0\leq p_+$ by comparison principle. Hence the last expression in \eqref{eq:the speed} holds. It follows from (iii) of Theorem \ref{thm:the sol} that
		 \begin{equation}\label{eq:Z=0}
		 	\{(t,h^{-1}(\omega t))| \omega\in[\omega_1,\omega_2], t>0\}\subset\text{int}\{(\tau,\xi)| Z(\tau,\xi)=0\},		\text{ where }  0<\omega_1<\omega_2<\lambda(0)
		 \end{equation}  	
		 \begin{equation}\label{eq:Z<0}
		 		\left\{(t,h^{-1}(\omega t))| \omega\in[\omega_3,\omega_4], t>\frac{h(0)}{\omega_3-\lambda(0)}\right\}\subset\{(\tau,\xi)| Z(\tau,\xi)<0\},		\text{ where }  \lambda(0)<\omega_3<\omega_4.
		 \end{equation}
		 Take $\omega<\lambda(0)$ and set $\kappa\leq\min\{\frac{\omega}{2}, \frac{\lambda(0)-\omega}{2}\}$.  	
		 Then one can easily  check that
		 \begin{equation*}
		 	\frac{(\omega-\kappa)t-h(0)}{\epsilon}+h(0)\leq\frac{\omega t}{\epsilon}\leq\frac{(\omega+\kappa)t-h(0)}{\epsilon}+h(0),\ \forall t\geq t_0:=\frac{h(0)}{\kappa}\text{ and } \epsilon\in(0,1),
		 \end{equation*}
		 which yields that $h^{-1}\left(\frac{\omega t}{\epsilon}\right)\in\left\{h^{-1}\left(\frac{\tilde \omega t-h(0)}{\epsilon}+h(0)\right)|\tilde \omega\in[\omega-\kappa,\omega+\kappa]\right\}$ for $t\geq t_0$.
		 Therefore, we have
		 \begin{equation}\label{eq:liminf>0}
		 \begin{split}
		 	&\quad\liminf\limits_{t\to\infty}u(t,h^{-1}(\omega t))
		 	  =\liminf\limits_{\epsilon\to0}u\left(\frac{t_0}{\epsilon},h^{-1}\left(\frac{\omega t_0}{\epsilon}\right)\right)\\
		 	&\geq\liminf\limits_{\epsilon\to0}\inf\limits_{\tilde\omega\in[\omega-\kappa,\omega+\kappa]}u\left(\frac{t_0}{\epsilon},h^{-1}\left(\frac{\tilde \omega t_0-h(0)}{\epsilon}+h(0)\right)\right)\\
		 	&=\liminf\limits_{\epsilon\to0}\inf\limits_{\tilde\omega\in[\omega-\kappa,\omega+\kappa]}u_\epsilon(t,h^{-1}(\tilde\omega t))>0.
		 \end{split}
		 \end{equation}  	
		 The last inequality follows from  Lemma \ref{lem:er fen} since $(t,h^{-1}(\tilde\omega t))\in\text{int}\{Z=0\}$ by \eqref{eq:Z=0}.
		  With \eqref{eq:liminf>0} at hand, we can prove as Step 3 in the proof of \cite[Theorem 1.2]{liang2021propagation} to deduce that 
		 $$ \lim\limits_{t \to +\infty}	 \sup\limits_{0 \leq x \leq h^{-1}(\omega t)}	 |u(t,x) - 1|  = 0, \text{ for all }\omega \in (0,\lambda(0)).$$
		 By similar arguments as above, we have
		  \begin{equation}\label{eq:limsup=0}
		 	 \limsup\limits_{t\to\infty}u(t,h^{-1}(\omega t))=0 \text{ for any }\omega>\lambda(0).
		 \end{equation}  	
		  On the other hand, taking the supersolution $\overline{u}(t,x)$ given in Lemma \ref{lem:well defined and initial value}, we have
		 $$\limsup\limits_{t\to\infty}\sup\limits_{\omega\in[\overline{c}+1,+\infty)}u(t,h^{-1}(\omega t))=0.$$
		 Combining this with \eqref{eq:limsup=0}, we have
		 $$ \lim\limits_{t \to +\infty}	 \sup\limits_{x \geq h^{-1}(\omega t)}	u(t,x) = 0, \text{ for all }\omega \in(\lambda(0),+\infty),$$
		 and thus \eqref{eq:ss} holds. 
		 Furthermore, the second expression in \eqref{eq:where is level set} can be obtained immediately.
	\end{proof}

	\subsection{Proof of the main results and examples}

In this section, we will first prove Theorems \ref{the:1} and \ref{the:2}, and then give some examples.
\begin{proof}[Proof of Theorem \ref{the:1}]
      Since $u_0\in\mathcal U_{p_0^-,p_0^+}$, it is known that (see the statements after (\textbf{A2})) 
      \begin{equation}\label{eq:local unif to 1}
      	\lim\limits_{t\to+\infty}\sup\limits_{|x|\leq R_0}u(t,x)=1.
      \end{equation}
      Hence we may consider the behavior of $u$ on $\{(t,x)| t>0, x\in[R_0,+\infty)\}$, or consider $v(t,x):=u(t,x+R_0)$ on $\{(t,x)| t>0, x\in[0,+\infty)\}$.
      That is, $v$ satisfies
      \begin{equation*}
      	\begin{cases}
      		v_t = \partial_x(a(x+R_0)\partial_{x}v) + b(x+R_0)\partial_x v+  f(x+R_0,v), \ &x \in \RR,\  t > 0,\\
      		v(0,x) = u_0(x+R_0)=e^{-\tilde h(x)}, \ &x \in \RR,
      	\end{cases}
      \end{equation*}
      where $\tilde h(x)=h(x+R_0)$ satisfies \eqref{eq:intial data} and \eqref{eq:h}.
      Note that the generalized principal eigenvalue $\lambda(p)$ of $L_p$ remains unchanged under translation of its coefficients.
      Hence, applying  Theorem \ref{the:1 simple case} to $v$, one has
      \begin{equation}\label{eq:ss of v}
      	\begin{cases}
      		\displaystyle{\lim_{t\rightarrow+\infty}}\sup \limits_{x\geq\omega t}|u(t,x+R_0)|=\displaystyle{\lim_{t\rightarrow+\infty}}\sup \limits_{x\geq\omega t}|v(t,x)|=0, \text{ for any }\omega>\omega^+,\\
      		\displaystyle{\lim_{t\rightarrow\infty}}\sup \limits_{0\leq x\leq \omega t}|u(t,x+R_0)-1|=\displaystyle{\lim_{t\rightarrow\infty}}\sup \limits_{0\leq x\leq \omega t}|v(t,x)-1|=0, \text{ for any }\omega<\omega^+,
      	\end{cases}
      \end{equation}
      where $\omega^+$ is still given by \eqref{eq:the speed}.
      Besides, \eqref{eq:ss} and \eqref{eq:where is level set} hold for $v$.
      Thus, Theorem \ref{the:1} follows immediately.
     \end{proof}
     
     \begin{proof}[Proof of Theorem \ref{the:2}]
     	By comparison principle, we only need to consider the case $u_0=e^{-h}$.
     	Let $\phi_1$ be eigenfunction as \eqref{eq:ef on 01}  on $[\frac{\delta}{2},2\delta]$.
     	Then it is easy to check that $u_1=\phi_1$ if $x\in[\frac{\delta}{2},2\delta]$, $=0$ otherwise is a subsolution. 
     	Therefore, $\liminf\limits_{t\to+\infty}u(t,\delta)>0$. 
     	On the other hand, similar to the homogenization method used earlier on $x >R_0$, we have
     	$\liminf\limits_{t\to+\infty}u\left(t,\omega t\right)>0$ for any $\omega\in(0,\omega^+)$  if $p^+_0>0$, and 
     	$\liminf\limits_{t\to+\infty}u\left(t,h^{-1}(\omega t)\right)>0$ for any $\omega\in(0,\lambda(0))$ if $p^+_0=0$.
     	Hence, the conclusions of the theorem can be obtained similarly to Theorem \ref{the:1 simple case}.
     \end{proof}
     
      \begin{proof}[Proof of Theorem \ref{thm:front like initial data}]
     	The comparison principle yields that $u\in[0,1]$.
     	Then all the conclusions of Theorem \ref{the:2} hold.
     	Below we show that $\delta$ can be taken as $-\infty$.  
     
     	Note that  the solution $u = u(t,x)$ converges to $1$ as $t\to+\infty$ locally in $x\in\mathbb R$, and $\inf\limits_{x\leq0}u_0(x)>0$. 
     	Then we set $m:=\min\{\inf\limits_{t\geq0}u(t,0),\inf\limits_{x\leq0}u_0(x)\}>0$ and $\underline{f}(u)=\frac{1}{2}\inf\limits_{x\in\RR}\partial_uf(x,0)u(\epsilon-u)$ with $\epsilon<m$ small such that $\underline{f}(u)\leq f(x,u)$ for any $(x,u)\in\RR\times[0,1].$
     	Then let $\underline{u}$ be the solution to 
     	$$\frac{d\underline{u}}{dt}=\underline{f}(\underline{u}), \text{ and }\underline{u}(0)=m/2.$$
     	Then by comparison principle on $\{t>0, x\leq0\}$, we have $u(t,x)\geq\underline u(t)$,
     	which yields that $\liminf\limits_{t\to+\infty}\inf\limits_{x\leq0}u(t,x)>0$.
     	Hence $\lim\limits_{t \to +\infty}	 \inf\limits_{ x\leq \omega t}	 u(t,x)>0  \text{ for any } \omega \in (0,\omega^+).$
     	Then by an argument similar to the proof of Theorem \ref{the:1 simple case}, we obtain $\lim\limits_{t \to +\infty}	 \sup\limits_{ x\leq \omega t}	| u(t,x)-1|=0  \text{ for any } \omega \in (0,\omega^+)$.	
     \end{proof}
     
     \begin{proof}[Proof of Corollary \ref{cor:one side}]
     	First of all, it follows from \cite[Section 3.6]{PW_MP} that $u>0$ for $t>0$ and $x>0$.
     	Then the conclusions follow from  Theorem \ref{the:2}.     	
     \end{proof}

 \begin{proof}[Proof of Theorem \ref{thm:general initial data}]
	 Since $u_0\le e^{-g}$, the upper bound for $E^+_\theta(t)$ follows directly from the comparison principle.
	 Below we prove that for some $x_0>0$ and $c_0\in\mathbb R$,
	 \begin{equation}\label{eq:low bound at t=1}
	 		 u(1,x) \ge e^{-h(x+x_0)+c_0}
	 \end{equation}
	 for all sufficiently large $x$. 
	Note that $ h(\cdot+x_0)-c_0 \text{ satisfies }	\eqref{eq:ass of h}$ for $x\geq R_0$ (enlarging $R_0$ if necessary)  and \eqref{eq:intial data} $\text { as }x\to+\infty 
	.$
	 Once this is established, taking $t=1$ as the initial time,  the comparison principle again yields the lower bound estimate for $E^+_\theta(t)$.
	 
	 Without loss of generality, assume that the intervals $I_n^+$ are ordered so that
	$ \sup I_n^+ < \inf I_{n+1}^+,$
	 and let $y_n$ denote the midpoint of $I_n^+$. Since $\{I_n^+\}$ is relatively dense in $\mathbb R_+$, for any $\xi \ge R_0 + L$, there exists some $n(\xi)$ such that
	 \begin{equation}\label{eq:the int}
	 I_{\xi,\kappa}:=[y_{n(\xi)}-\kappa,y_{n(\xi)}+\kappa]\subset\subset I_{n(\xi)}^+ \subset \left( \xi - \frac{L}{2},\, \xi + \frac{L}{2} \right),
	 \end{equation}
	 where $\kappa, L$ are the constants in \eqref{eq:interval_conditions} and \eqref{eq:relatively_dense_plus}, respectively.
	If there are several of such intervals; we simply fix one of them.
	 We consider 
	 \begin{equation}\label{eq:low Dirichlet}
	 	\begin{cases}
	 	v_{t}=\partial_x(a^\xi\partial_xv)+b^\xi\partial_xv, \ &t\in(0,1],\ x\in(-L,L),\\
	 	v(t,x)=0,\ &t\in(0,1],\ x\in\{\pm L\},\\
	 	v(0,x)=v^\xi_0(x), \  &x\in[-L,L],\\
	 \end{cases}
	 \end{equation}
	where $a^\xi=a(\cdot+\xi)$ and $b^\xi=b(\cdot+\xi)$, and $v^\xi_0$ is continuous and satisfies
	\[\mathds{1}_{I_{\xi,\kappa}-\xi}(x)e^{-h(x+\xi)}\leq v^\xi_0(x)\leq \mathds{1}_{I^+_{n(\xi)}-\xi}(x)e^{-h(x+\xi)}, \ x\in[-L,L],\]
	here the set $A-\xi:=\{x-\xi\mid x\in A\}$.
	Let $G^\xi(t,x,y)$ denote the Dirichlet heat kernel associated with the operator in \eqref{eq:low Dirichlet}. 
	Then $ v^\xi(t,x)=\int_{(-L,L)}G^\xi(t,x,y)v^\xi_0(y)dy$, and for all $x,y\in(-L,L)$ and $t>0$, by the lower bound estimate (see, e.g., \cite[Theorem 1.1]{cho2012two-side estimate}), we have
	\[
	G^\xi(t,x,y) \ge \frac{c_1}{\sqrt t}
	\min\left\{1,\frac{\rho(x)}{\sqrt t}\right\}
	\min\left\{1,\frac{\rho(y)}{\sqrt t}\right\}
	\exp\left(-\frac{c_2 (x-y)^2}{t}\right),
	\]
	where $c_1,c_2>0$ are constants independent of $\xi$, and
	\[
	\rho(x) := \operatorname{dist}(x,\partial(-L,L)) = \min\{|x-L|,\, |x+L|\}.
	\]
	Hence 
	 \begin{equation}\label{}
		\begin{split}
			&\quad v^\xi(1,0)=\int_{(-L,L)}G^\xi(1,0,y)v^\xi_0(y)dy\geq\int_{I_{\xi,\kappa}-\xi}G^\xi(1,0,y)dy\inf\limits_{y\in I_{\xi,\kappa}-\xi}v^\xi_0(y)\\
			&\geq C_0\inf\limits_{y\in I_{\xi,\kappa}-\xi}e^{-h(y+\xi)}\geq C_0\inf\limits_{y\in[-L/2,L/2]}e^{-h(y+\xi)}\geq C_0e^{-h(\xi+L/2)}.
		\end{split}
	\end{equation}  	
	Here $C_0>0$  depending only on $L$ and $\kappa$, but not on $\xi$. The penultimate inequality follows from \eqref{eq:the int}, and the last inequality is a consequence of the monotonicity of $h$ on $[R_0,\infty)$.
	Observe that $v^\xi(t,x)$ is a subsolution of $u(t,x+\xi)$ on $(0,1]\times[-L,L]$. Therefore, by the comparison principle,
	\[
	u(1,\xi) \ge v^\xi(1,0) \ge C_0 e^{-h(\xi + L/2)} \text{ for all }\xi \ge R_0+L.
	\]
	Hence \eqref{eq:low bound at t=1} holds with $x_0=L/2$ and $c_0=\ln C_0$.	    	
\end{proof}
     
        Next, we give   four types of typical examples for the case where $p^+_0=0$, which were considered by Hamel and Roques \cite{hamel2010fast}.
         We always assume that the initial data \(u_0 \) satisfies the condition in Theorem \ref{thm:front like initial data}.
        Then $$\min E_\theta(t) \sim \max E_\theta(t) \sim h^{-1}(\lambda(0)t) \text{ as } t \to +\infty,$$ or, in other words,
        $h(\min E_\theta(t)) \sim h(\max E_\theta(t)) \sim (\lambda(0)t) \text{ as } t \to +\infty.$
        
         If \[u_0(x) \sim C (\ln x)^{-\alpha} \quad \text{as } x \to +\infty\] for some $\alpha, C>0$, then
        $\ln(\ln(\min E_\theta(t))) \sim \ln(\ln(\max E_\theta(t))) \sim \frac{\lambda(0)t}{\alpha} \text{ as } t \to +\infty.$
       
         If \[u_0(x) \sim Cx^{-\alpha} \quad \text{as } x \to +\infty\] for some $\alpha, C>0$, then   $\min E_\theta(t) \sim \max E_\theta(t) \sim e^{\frac{\lambda(0)t}{\alpha}} \text{ as } t \to +\infty.$
       
       If \[u_0(x) \sim C e^{-\beta x^\alpha} \quad \text{as } x \to +\infty\] for some $\alpha\in(0,1)$ and $\beta, C>0$, then 
        $\min E_\theta(t) \sim \max E_\theta(t) \sim \lambda(0)^{1/\alpha} \left(\frac{t}{\beta}\right)^{1/\alpha} \text{ as } t \to +\infty.$

      If  
            \[u_0(x) \sim Ce^{-\alpha x / \ln x} \quad \text{as } x \to +\infty\]
      for some $\alpha, C>0$, then $\frac{\alpha\min E_\theta(t)}{\ln\min E_\theta(t)} \sim \frac{\alpha\max E_\theta(t)}{\ln\max E_\theta(t)} \sim \lambda(0)t \text{ as } t \to +\infty$.
      Hence $$\min E_\theta(t) \sim \max E_\theta(t) \sim \frac{\lambda(0)t\ln t}{\alpha}  \text{ as } t \to +\infty,$$ which is consistent with the result in the homogeneous case \cite{hamel2010fast}.

       An interesting question is whether, for an initial condition with sub-exponential decay, the expansion speed of its level set exhibits lower-order corrections analogous to those in Bramson's results—and if so, how to determine them. 
       In the context of linear propagation speeds, it is known that such corrections are of order $\ln t$ (\cite{Alfaro2025MathAnn,bramson83MAMS,hamel2013NHM,hamel2016JEMS,zhang2025arxiv}). 
       However, the four examples presented above indicate that any potential lower-order corrections, if they exist, are highly intricate and inevitably vary depending on the specific type of sub-exponential decay of the initial data. We leave this issue for future investigation.

	\clearpage
\end{document}